\documentclass[10pt]{amsart}

\usepackage{algorithm}
\usepackage{algorithmic}

\usepackage{graphicx}
\usepackage{rotating}

\usepackage{color}
\usepackage{bm}
\usepackage{amsthm}
\usepackage[usenames,dvipsnames,svgnames,table]{xcolor}

\usepackage{amssymb, amsfonts}
\usepackage{latexsym}
\usepackage{bm}
\usepackage{color}
\usepackage{array}
\usepackage{mathrsfs}
\usepackage{amsmath}
\usepackage{flafter}

\definecolor{red}{rgb}{1,0,0}

\def\v{{\bm v}}
\def\u{{\bm u}}

\def\f{{\bm f}}

\def\ep{{\bm \epsilon}}
\def\sig{{\bm \sigma}}

\def\bU{{\boldsymbol U}}
\def\bV{{\boldsymbol V}}
\def\bu{{\boldsymbol u}}

\renewcommand{\div}{\operatorname{div}}
\newcommand{\Lbrack}{\lbrack\!\lbrack} 
\newcommand{\Rbrack}{\rbrack\!\rbrack} 
\newcommand{\R}[1]{I\!\!R^{#1}}
\newcommand{\vek}[1]{\boldsymbol{#1}}

\newtheorem{lemma}{Lemma}
\newtheorem{theorem}{Theorem}
\usepackage{verbatim}

\newtheorem{corollary}[theorem]{Corollary}

\newtheorem{remark}[theorem]{Remark}

\usepackage{epsfig}
\usepackage{subfigure}

\def\Xi{{X}_h}

\begin{document}

\title[Parameter-robust stability of Biot's model]{Parameter-robust stability of classical three-field formulation of Biot's consolidation model}

\author[Qingguo Hong and Johannes Kraus]{Qingguo~Hong and Johannes~Kraus}

\address{Department of Mathematics,
        Pennsylvania State University, State College, PA 16802, U.S.A.}
\email{huq11@psu.edu}

\address{Faculty of Mathematics,
University of Duisburg-Essen,
Thea-Leymann-Str.~9, 45127 Essen, Germany}
\email{johannes.kraus@uni-due.de}



\maketitle

\begin{abstract}
This paper is devoted to the stability analysis of a classical three-field formulation of Biot's consolidation
model where the unknown variables are the displacements, fluid flux (Darcy velocity), and pore pressure.
Specific parameter-dependent norms provide the key {in establishing} the full parameter-robust inf-sup
stability
of the continuous problem. {Therefore,} {stability results presented here are uniform not only} with
respect to the Lam\'e
parameter $\lambda$, but also with respect to all the other model parameters. This allows for the
construction of a uniform block diagonal preconditioner {within} the framework of operator preconditioning.
Stable discretizations that meet the required conditions for full robustness and
%preserve
{guarantee} mass conservation,
%exactly
{both locally and globally,}
%and globally
%on a discrete level
are discussed and corresponding optimal error
estimates proven.
\end{abstract}

%\section{Introduction: Three-field formulation of Biot's consolidation model}
\section{Introduction: Biot's consolidation model}

Poroelastic models describe mechanical deformation and fluid flow in porous media. They have a wide
range of applications in medicine, biophysics and geosciences such as the computation of intracranial pressure,
trabecular bone stiffness under different loading conditions, reservoir simulation, waste repository performance,
${\rm CO}_2$ sequestration, consolidation of soil under surface loads, subsidence due to fluid withdrawal and
many others, see, e.g.~\cite{Smith2007interstitial,Stoverud2016poroelastic,Wang2000theory,Detournay1993fundamentals}.
 
A classical and widely used model has been introduced by Biot~\cite{Biot1941general,Biot1955theory} and is
based on the following assumptions:
\begin{itemize}
\item[(i)] the porous medium is saturated by fluid and the temperature is constant,
\item[(ii)] the fluid in the porous medium is (nearly) incompressible,
\item[(iii)] the solid skeleton (matrix) is formed by an elastic material and deformations and strains are relatively small,
\item[(iv)] the fluid flow is driven by Darcy's law (laminar flow).
\end{itemize}

For homogeneous isotropic linear elastic porous media, the Biot model in an open domain
$\Omega \subset \R{d}$, $d=2,3$, comprises the following system of partial differential equations ({PDEs}):
\begin{subequations}\label{eq:Biot}
\begin{eqnarray}
- \div \sig +  c_{up} \nabla p &=& \boldsymbol f \qquad \mbox{in } \Omega \times (0,T), \label{Biot1} \\
\boldsymbol v &=& -K \nabla p \qquad \mbox{in } \Omega \times (0,T), \label{Biot2} \\
c_{pu} \div \dot{\boldsymbol u} - \div \boldsymbol v  + c_{pp} \dot p &=& g \qquad \mbox{in } \Omega \times (0,T), \label{Biot3} \\
\ep(\u)&=&\displaystyle \frac{1}{2}(\nabla \u+(\nabla \u)^T), \label{compatibility} \\
\sig &=& 2 \mu \ep(\u) + \lambda \div (\u) \boldsymbol I. \label{constitutive}
\end{eqnarray}
\end{subequations}
Here $\lambda$ and $\mu$ denote the Lam\'e parameters which are defined by 
$$
\lambda:=\frac{\nu E}{(1+\nu)(1-2\nu)}, \quad
\mu:=\frac{E}{2(1+\nu)}
$$
in terms of the modulus of elasticity (Young's modulus) $E$ and the Poisson ratio $\nu\in[0,1/2)$.

The constant $c_{up}=c_{pu}=\alpha$ coupling the pore pressure $p$ and the displacement variable $\u$
is the Biot-Willis constant, $K$ is the hydraulic conductivity, given by the quotient between the permeability
of the porous medium $\kappa$ and the viscosity of the fluid $\eta$;
$\boldsymbol I$ denotes the identity tensor and $\sig$ and $\ep$ are the effective stress and strain tensors,
respectively, which are related to each other via the constitutive equation~\eqref{constitutive}; The strain tensor
 $\ep(\u)$ is given by the symmetric part of the gradient of the displacement field as defined in the compatibility
condition~\eqref{compatibility}.
The time derivatives of $\u$ and $p$ in the continuity equation~\eqref{Biot3} are denoted by $\dot{\u}$ and $\dot p$.
Finally, $\v$ denotes the fluid flux, sometimes also called percolation velocity of the fluid, which is assumed to
be proportional to the (negative) pressure gradient as expressed by Darcy's law~\eqref{Biot2}.
The right hand side $\boldsymbol f$ in the equilibrium equation~\eqref{Biot1} represents the density of the applied
body forces and the source term $g$ in~\eqref{Biot3} a forced fluid extraction or injection.

The system~\eqref{eq:Biot} is completed by proper boundary and initial conditions, e.g.,
\begin{subequations}\label{eq:Biot_BC}
\begin{eqnarray}
p(\vek{x},t) &=& p_D(\vek{x},t)  \qquad \mbox{for } \vek{x} \in \Gamma_{p,D}, \quad t > 0 , \\ 
\v(\vek{x},t) \cdot {\bm n} (\vek{x}) &=& q_N(\vek{x},t)  \qquad \mbox{for }  \vek{x} \in \Gamma_{p,N}, \quad t > 0 , \\ 
\u(\vek{x},t) &=& {\bu}_D(\vek{x},t)  \qquad \mbox{for }  \vek{x} \in \Gamma_{\vek{u},D}, \quad t > 0 , \\ 
{\bm \sigma(\vek{x},t)} \, {\vek{n}} (\vek{x}) &=& {\bm g}_N(\vek{x},t)  \qquad \mbox{for }  \vek{x} \in \Gamma_{\vek{u},N}, \quad t > 0 ,
\end{eqnarray}
%\begin{eqnarray}
%p(\boldsymbol x,t)=p, &\quad \sig \boldsymbol n &= \boldsymbol h \qquad \mbox{on } \Gamma_1, \\
%\u = \boldsymbol 0, &\qquad K (\nabla p) \cdot \boldsymbol n &= 0  \, \qquad \mbox{on } \Gamma_2,
%\end{eqnarray}
\end{subequations}
where
$\Gamma_{p,D} \cap \Gamma_{p,N} = \emptyset$,
$\overline{\Gamma}_{p,D}\cup \overline{\Gamma}_{p,N}=\Gamma=\partial{\Omega}$
and
$\Gamma_{\vek{u},D} \cap \Gamma_{\vek{u},N} = \emptyset$,
$\overline{\Gamma}_{\vek{u},D} \cup \overline{\Gamma}_{\vek{u},N}=\Gamma$;
Initial conditions at the time $t=0$ to complement the boundary conditions~\eqref{eq:Biot_BC}
are given by
\begin{subequations}\label{eq:Biot_IC}
\begin{eqnarray}
p(\vek{x},0) &=& p_0(\vek{x}) \qquad \vek{x} \in \Omega, \\
\u (\vek{x},0) &=& \u_0(\vek{x}) \qquad \vek{x} \in \Omega.
\end{eqnarray}
\end{subequations}

Making use of the constitutive equation~\eqref{constitutive} to eliminate the stress variable from the system~\eqref{eq:Biot}
results in the classical three-field formulation of the Biot model.

A  common way to solve the time-dependent problem numerically, is then to discretize it in time and solve a static problem
in each time step. Using the backward Euler method for time discretization, in this way one obtains a three-by-three block
system of time-step equations
\begin{equation}\label{eq:1}
\mathcal{A}\left[\begin{array}{c}\u^k \\ \v^k \\ p^k
\end{array}\right]=
\left[\begin{array}{c} \f^k \\ {\bm 0} \\ \tilde{g}^k \end{array}\right]
\end{equation}
where
\begin{equation}\label{eq:2}
\mathcal{A}:=\left[\begin{array}{ccc}
-2 \mu \div \ep-\lambda \nabla \div & 0 & c_{up}\nabla \\
0 & \tau K^{-1} & \tau \nabla \\
-c_{pu} \div & -\tau \div & -c_{pp} I  \end{array}\right]
\end{equation}
for the unknown time-step functions
\begin{subequations}\label{eq:spaces}
\begin{eqnarray*}
\u^k=\u(\vek{x},t_k) \, \in \, \bU &:=& \{ \u \in H^1(\Omega)^d:  \u={\u}_D \mbox{ on } \Gamma_{\u,D} \} , \\
\v^k=\v(\vek{x},t_k) \, \in \, \bV &:=&  \{ \v \in H(\div,\Omega):  \v \cdot \vek{n} = q_N \mbox{ on } \Gamma_{p,N} \}, \\
p^k=p(\vek{x},t_k) \, \in \, P &:=& L^2(\Omega),
\end{eqnarray*}
\end{subequations}
and the right hand side time-step functions $\f^k=\f(\vek{x},t_k)$ and
$\tilde{g}^k=-\tau g(\vek{x},t_k) - c_{pu} \div (\u^{k-1}) - c_{pp} p^{k-1}$
at any given time moment $t=t_k=t_{k-1}+\tau$.
%
%{Here $H^1(\Omega)^d$ and $L^2(\Omega)$ denote the space of vector-valued $H^1$-functions and
%the space of square Lebesgue integrable functions, respectively. Moreover, the space $H(\div ;\Omega)$ is
%defined as
%$$
%H(\div ;\Omega):=\{\bm v \in L^2(\Omega): \div  \bm v\in L^2(\Omega)\}.
%$$
%The norms $\|\cdot\|_1$ and $\|\cdot\|_{\div}$ are defined as $\|\bm u\|^2_1=\|\bm u\|^2+\|\nabla \bm u\|^2$ and
%$\|\bm v\|^2_{\div}:=\|\bm v\|^2+\|\div \bm v\|^2$.}
%
{As in the remainder of this paper we will consider the static problem~\eqref{eq:1}--\eqref{eq:2},
for convenience, we will drop the superscript for the time-step functions, that is, denote $\u^k$, $\v^k$
and $p^k$ by $\u$, $\v$ and $p$, respectively.}

{
%We will use the following standard Sobolev spaces and norms.
Following the standard notation, $L^2(\Omega)$ denotes the space of square Lebesgue
integrable functions equipped with the standard $L^2$ norm $\|\cdot\|$ and $H^1(\Omega)^d$ the
space of vector-valued $H^1$-functions equipped with the norm $\| \cdot \|_1$ defined by
$\|\bm u\|^2_1:=\|\bm u\|^2+\|\nabla \bm u\|^2$. Moreover, %we will refer to the space
$
H(\div ;\Omega):=\{\bm v \in L^2(\Omega)^d: \div  \bm v\in L^2(\Omega)\}
$
where the standard Sobolev norm $\|\cdot\|_{\div}$ is defined by
$\|\bm v\|^2_{\div}:=\|\bm v\|^2+\|\div \bm v\|^2$.
}

Frequently we will consider the case in which $\Gamma_{\u,D}=\Gamma_{p,N}=\Gamma$ and ${\u}_D={\bm 0}$, $q_N=0$
in which we write $\bU=H^1_0(\Omega)^d$ and  $\bV=H_0(\div, \Omega)$. In order to determine the solution for the
pressure variable $p$ uniquely one can set $P= L^2_0(\Omega):=\{ p \in L^2(\Omega) : \int_{\Omega} p \, d\vek{x} = 0 \}$. 

In many applications the variations of the model parameters are quite large. In geophysical applications, the permeability
typically varies in the range from $10^{-9}$ to $10^{-21} m^2$ whereas Young's modulus is typically in the order of GPa
and the Poisson ratio in the range $0.1-0.3$, see~\cite{Wang2000theory,Lee2016parameter,Coussy2004poromechanics}.
Soft tissue of  the central nervous system on the other hand has a permeability of about $10^{-14}$ to $10^{-16} m^2$
whereas Young's modulus is typically in the order of kPa and the Poisson ratio in the range $0.3$ to almost $0.5$, see
\cite{Smith2007interstitial,Stoverud2016poroelastic}. For that reason it is important that not only the formulation of the problem
but also the numerical methods for its solution are stable over the whole range of values of the parameters in the model.

The stability of the time discretization has been studied in~\cite{Axelsson2012stable} and will not be addressed here.
Instead we will focus on the issue of inf-sup stable finite element discretizations of the static problem~\eqref{eq:1}--\eqref{eq:2}.
It is a well known fact that the LBB condition, see~\cite{Babuska1971error,Brezzi1974existence}, plays the crucial role
in  the well-posedness analysis of the continuous problem and its discrete counterparts arising from mixed
finite element discretizations. It is also the key tool in deriving a priori error estimates. Inf-sup stability for the Darcy problem
as well as for the Stokes and linear elasticity problems is well understood and various stable mixed discretizations of either
of these systems of {PDEs} have been proposed over the years, see, e.g.~\cite{Boffi2013mixed} and the references therein.

Biot's model of poroelasticity combines theses equations and the parameter-robust stability of its three-field
formulation becomes more delicate as we will see in the next sections. Alternative formulations that can be proven to
be stable independently of the model parameters (in certain norms) include a two-field formulation for the displacements
and pore pressure, see~{\cite{Lee2016parameter,adler2017robust}}, and a new three-field formulation that--besides the
displacements--{introduces} two pressure unknowns, one for the fluid pressure and one for the total pressure defined
as a weighted sum of fluid and solid pressure \cite{Lee2016parameter}.

Compared to the new three-field formulation presented in~\cite{Lee2016parameter}, the classic three-field formulation
of Biot's consolidation model keeps Darcy's law {in order to guarantee} fluid mass conservation.
Recently, a four-field formulation has been proposed in which the stress tensor is kept as a variable in the system,
see~\cite{Lee2016robust}, and the error analysis there is robust with respect to $\lambda$, but not uniform with
respect to the other parameters such as $\tau$ and $K$. 

{Nonconforming finite elements have been shown to be beneficial with regard to reducing pressure oscillations}
in computations based on the classical three-field formulation, see~\cite{Hu2017nonconforming}. The lowest %possible
approximation order, consisting of Crouzeix-Raviart finite elements for the displacements, lowest-order Raviart-Thomas
elements for the Darcy velocity, and a piecewise constant approximation of the pressure unknown, in combination with
a mass-lumping technique for the Raviart-Thomas elements results in a computationally efficient method. However, the
norms defined in this {work}, and in many others, do not {allow to establish} the full parameter-robust stability
for which we aim.

%%{In this work, by introducing proper norms which play crucial role in our analysis, we proved the parameter-robust
%%stability of the classic three-field formulation of Biot's consolidation model.
In the present paper, we establish this {full} parameter-robust stability for the classic three-field formulation of Biot's
consolidation model. Crucial in our analysis is the definition of proper norms for which we prove that the constants in
the related inf-sup conditions do not depend on any of the model parameters.
Further, we propose a discretization that preserves fluid mass conservation {at a} discrete level. We also prove
the full parameter-robust stability of the discretized problem and of course the related optimal error estimates.
The remainder of the paper is organized as follows.

In Section~\ref{sec:com_unstable_disc}, we briefly revisit non-uniform stability results, and
%with respect to certain natural norms, get a motivation,
{make some useful observations which motivate the subsequent analysis.}
In Section~\ref{sec:par_rob_stab_model} we {introduce} the parameter-dependent norms based on which we establish
the parameter-robust stability of the weak formulation of the continuous problem~\eqref{eq:1}--\eqref{eq:2}.
Section~\ref{sec:uni_stab_disc_model} analyzes mixed finite element discretizations that provide discrete parameter-robust
inf-sup stability and full fluid mass conservation.
Applying the theory of operator preconditioning, see~\cite{Mardal2011preconditioning}, the results from
{Sections~\ref{sec:par_rob_stab_model} and~\ref{sec:uni_stab_disc_model}}
{imply the uniformity (parameter-robustness) of the (canonical) norm-equivalent block-diagonal preconditioners.}
In Section~\ref{sec:error_estimates} we use our findings to derive robust optimal a priori error estimates.
Finally, Section~\ref{conclusion} gives some {concluding remarks}. 

Throughout this paper, the hidden constants in $\lesssim,~\gtrsim$~and $\eqsim$ are 
independent of the parameters $\mu,\lambda,c_{up},\tau, K, c_{pu},c_{pp}$ and the mesh size $h$.
Hence the hidden constants in $\lesssim,~\gtrsim$~and $\eqsim$ are independent of 
$\lambda, R_p, \alpha_p$ and the mesh size $h$.

\section{A revisit of non-uniform stability results}\label{sec:com_unstable_disc}
We {begin} our stability analysis with recasting the equations~\eqref{eq:1}--\eqref{eq:2}.
We first divide all the parameters in the model~\eqref{eq:1}--\eqref{eq:2} by $2 \mu$ herewith eliminating the parameter $\mu$.
That is, we make the substitutions
$$2\mu \rightarrow 1, \ \lambda/2\mu\rightarrow \lambda, \ c_{up}/2\mu\rightarrow c_{up}, \
\bm f/2\mu\rightarrow \bm f, \ \tau/2\mu\rightarrow \tau, \ c_{pp}/2\mu \rightarrow c_{pp}, \ g/2\mu \rightarrow g.$$
Herewith, equation \eqref{eq:1} becomes 
\begin{subequations}\label{eq:3}
\begin{eqnarray}
-\div \ep (\boldsymbol u)-\lambda \nabla \div \boldsymbol u +  c_{up} \nabla p&=&\boldsymbol f \label{eq:3a} , \\
  \tau K^{-1} \boldsymbol v+ \tau \nabla p&=&\boldsymbol 0 ,  \label{eq:3b} \\
-c_{pu} \div \boldsymbol u  -\tau \div \boldsymbol v  -c_{pp} p&=&g . \label{eq:3c}
\end{eqnarray}
\end{subequations}
Now let $\tilde{\boldsymbol u}=c_{pu}\boldsymbol u, \tilde {\boldsymbol v}=\tau \boldsymbol v, \tilde p=c_{pu}^2 p, \tilde{\boldsymbol f}=c_{pu}\boldsymbol f$
and divide equation \eqref{eq:3b} by $\tau$ to get 
\begin{subequations}\label{eq:4}
\begin{eqnarray}
-\div \ep (\boldsymbol {\tilde u})-\lambda \nabla \div \boldsymbol {\tilde u} + \nabla \tilde p&=\boldsymbol {\tilde f} , \label{eq:4a} \\
%%c^2_{pu}
{\alpha^2}
\tau^{-1} K^{-1} \boldsymbol {\tilde v}+  \nabla {\tilde p}&=\boldsymbol 0 , \label{eq:4b} \\
-\div \boldsymbol {\tilde u}  -\div \boldsymbol {\tilde v}  -c_{pp} 
%%c^{-2}_{pu}
{\alpha^{-2}}
\tilde p&=g  ,\label{eq:4c}
\end{eqnarray}
\end{subequations}
where we have also used that {$c_{up}=c_{pu}=\alpha$}.
For convenience we denote 
$$R^{-1}_p={\alpha^2}\tau^{-1} K^{-1}, \quad \alpha_p=c_{pp} \, {\alpha^{-2}},$$
and further assume that the range of the parameters is
%$1\leq\lambda\leq \infty, 0\leq R^{-1}_p\leq \infty, 0\leq \alpha_p \leq \infty$. These assumptions are very general and reasonable. 
% In the subsequent, in order to simplify the notation, we skip the ``tilde'' symbol, i.e.,
% $\boldsymbol {\tilde u} \rightarrow \boldsymbol u, \boldsymbol {\tilde v}\rightarrow \boldsymbol v, \boldsymbol {\tilde f}\rightarrow \boldsymbol f, \tilde p\rightarrow p$,
$$
{\lambda \ge 1, \quad  R^{-1}_p \ge 0, \quad \alpha_p \ge 0.}
$$
%Finally obtain the system
%\begin{subequations}\label{eq:5}
%\begin{eqnarray}
%-\div \ep (\boldsymbol u)-\lambda \nabla \div \boldsymbol  u + \nabla  p&=\boldsymbol f \label{eq:5a}\\
%  R_p^{-1}\boldsymbol  v+  \nabla  p&=\boldsymbol 0 \label{eq:5b} \\
%-\div \boldsymbol u  -\div \boldsymbol v  -\alpha_p  p&=g \label{eq:5c}
%\end{eqnarray}
%\end{subequations}
%
These assumptions are very general and reasonable. 

In the subsequent, in order to simplify the notation, we skip the ``tilde'' symbol, i.e., we make the substitutions
$\boldsymbol {\tilde u} \rightarrow \boldsymbol u, \boldsymbol {\tilde v}\rightarrow \boldsymbol v, \boldsymbol {\tilde f}\rightarrow \boldsymbol f, \tilde p\rightarrow p$,
{and finally write the system, without loss of generality, in the form}
\begin{subequations}\label{eq:5}
\begin{eqnarray}
-\div \ep (\boldsymbol u)-\lambda \nabla \div \boldsymbol  u + \nabla  p&=\boldsymbol f {,} \label{eq:5a}\\
  R_p^{-1}\boldsymbol  v+  \nabla  p&=\boldsymbol 0 {,} \label{eq:5b} \\
-\div \boldsymbol u  -\div \boldsymbol v  -\alpha_p  p&=g,\label{eq:5c}
\end{eqnarray}
\end{subequations}
{{or, in short notation,} 
\begin{equation}\label{eq:operator:A }
A\left[\begin{array}{c}\u \\ \v \\ p
\end{array}\right]=
\left[\begin{array}{c} \bm f \\ {\bm 0} \\ \tilde{g} \end{array}\right]
\end{equation}
where
\begin{equation}\label{operator:A}
A:=\left[\begin{array}{ccc}
-\div \ep-\lambda \nabla \div & 0 & \nabla \\
0 & R_p^{-1}I & \nabla \\
-\div & - \div & -\alpha_{p} I  \end{array}\right].
\end{equation}
}
The weak formulation of \eqref{eq:5a}--\eqref{eq:5c} reads:
Find $(\boldsymbol u, \boldsymbol v, p)\in \boldsymbol U\times\boldsymbol V\times P$, such that for any
$(\boldsymbol w, \boldsymbol z, q)\in \boldsymbol U\times\boldsymbol V\times P$
\begin{subequations}\label{eq:8}
\begin{eqnarray}
(\ep (\boldsymbol u), \ep (\boldsymbol w)) +\lambda ( \div \boldsymbol  u, \div \boldsymbol  w)- ( p, \div \boldsymbol  w)&=&(\boldsymbol f, \boldsymbol w),\label{eq:8a}\\
(R_p^{-1}\boldsymbol  v, \boldsymbol  z)- (p, \div \boldsymbol  z)&=& 0, \label{eq:8b} \\
-(\div \boldsymbol u,q)  -(\div \boldsymbol v,q)  -\alpha_p  (p,q)&=&(g,q). \label{eq:8c}
\end{eqnarray}
\end{subequations}

Motivated by the work~\cite{lipnikov2002numerical},
{let us first consider} the Hilbert spaces $\bU=H^1_0(\Omega)^d, \bV=H_0(\div, \Omega),$ and $P= L^2_0(\Omega)$
with the {\it natural norms} defined by
\begin{eqnarray}
(\boldsymbol u, \boldsymbol w)_{\bU}&=&(\ep (\boldsymbol u), \ep (\boldsymbol w))+\lambda(\div \boldsymbol u, \div \boldsymbol w),\label{eq:norm4n}\\
(\boldsymbol v, \boldsymbol z)_{\bV}&=&(R^{-1}_p \boldsymbol v,\boldsymbol z)+R_p^{-1} (\div \boldsymbol v, \div \boldsymbol z),\label{eq:norm5n}\\
(p, q)_P&=&(p,q).\label{eq:norm6n}
\end{eqnarray}

Before we study the Biot's equations, we recall the  following well known results,
see, e.g.~\cite{ Brezzi1974existence,Boffi2013mixed}.
%{By the definitions~\eqref{eq:7a}--\eqref{eq:7c}, we have the following well known result, see,
%e.g.~\cite{ Brezzi1974existence,Boffi2013mixed}}
\begin{lemma}\label{Hdiv:inf-sup}
{There {exists} a positive constant $\beta_v>0$, such that 
\begin{equation}
\inf_{q\in P}\sup_{\bm v\in \bm V}\frac{(\div \bm v, q)}{\|\bm v\|_{\div}\|q\|}\geq \beta_d.
\end{equation}}
\end{lemma}
{
\begin{lemma}\label{Stokes:inf-sup}
There {exists} a positive constant $\beta_s>0$, such that 
\begin{equation}
\inf_{q\in P}\sup_{\bm u\in \bm U}\frac{(\div \bm u, q)}{\|\bm u\|_{1}\|q\|}\geq \beta_s.
\end{equation}
\end{lemma}
}
Let us now turn to the stability of the formulation~\eqref{eq:8a}--\eqref{eq:8c}. {We consider two
different cases. In the first case, for some nearly incompressible materials,  we have that the Lam\'e
parameter $\lambda$ {tends} to infinity. {If} we assume that $\lambda \gg 1$, $R_p^{-1} \eqsim 1$ 
and $0\leq\alpha_p \lesssim 1$ {then, defining the norms according to}~\eqref{eq:norm4n}--\eqref{eq:norm6n},
the boundedness of both
$(\ep (\boldsymbol u), \ep (\boldsymbol w)) +\lambda ( \div \boldsymbol  u, \div \boldsymbol  w)+(R_p^{-1}\bm v, \bm v)$
and $(\div \boldsymbol u,q) +(\div \boldsymbol v,q)$ are obvious.
Further, {for} $0 \leq \alpha_p \lesssim 1$, we obtain the boundedness of $\alpha_p  (p,q)$.
{Moreover, defining the norms by~\eqref{eq:norm4n}--\eqref{eq:norm6n}, using Lemma~\ref{Hdiv:inf-sup},
and choosing $(\bm u, \bm v)=(\bm 0, \bm v)\in \bm U\times \bm V$,
we obtain} the inf-sup condition  
\begin{equation}
\inf_{q\in P}\sup_{(\bm u, \bm v)\in \bm U\times \bm V}\frac{(\div \bm u, q)+(\div \bm v, q)}{(\|\bm u\|_{\bm U}+\|\bm v\|_{\boldsymbol V})\|q\|_P}\geq \beta_{dv}>0.
\end{equation}
Finally,  the coercivity of
$(\ep (\boldsymbol u), \ep (\boldsymbol u)) +\lambda (\div \bm u, \div \bm u)+(R_p^{-1}\bm v, \bm v)$
on the kernel set 
\begin{equation}\label{kernel}
\bm Z=\{(\bm u, \bm v)\in \bm U\times\bm V: (\div \bm u, q)+(\div\bm  v, q)=0, \forall q\in P\}
\end{equation}
can also be verified. In fact, since $(\bm u, \bm v)\in \bm Z$ means $\div \bm v=- \div\bm u$, it follows that
\begin{equation}\label{Hdiv:coercivity}
\begin{split}
&(\ep (\boldsymbol u), \ep (\boldsymbol u)) +\lambda(\div \bm u, \div \bm u)+(R_p^{-1}\bm v, \bm v)\\
&\geq (\ep (\boldsymbol u), \ep (\boldsymbol u)) +\frac{\lambda}{2}(\div \bm u, \div \bm u)+(R_p^{-1}\bm v, \bm v)+\frac{\lambda}{2}(\div \bm u, \div \bm u)\\
&= (\ep (\boldsymbol u), \ep (\boldsymbol u)) +\frac{\lambda}{2}(\div \bm u, \div \bm u)+(R_p^{-1}\bm v, \bm v)+\frac{\lambda}{2}(\div \bm v, \div \bm v)  ~~~(\hbox{by}~~   \div \bm v=- \div\bm u) \\
& \gtrsim  \frac{1}{2} \Big((\ep (\boldsymbol u), \ep (\boldsymbol u)) +\lambda(\div \bm u, \div \bm u)+(R_p^{-1}\bm v, \bm v)+(R_p^{-1}\div\bm v, \div\bm v)\Big),
\end{split}
\end{equation}}
where the last inequality comes from the assumption $\lambda \gg 1$, $R_p^{-1} \eqsim 1$.

{On the other hand, in many practical applications one has
$\lambda \eqsim 1$, $c_{up}\eqsim1$},
$c_{pp}=c_{up}/{\lambda}$,
and $K ~\hbox {or}~\tau$ tending to zero. 
{Then, since in this case} we have
%$\lambda=1, 0\leq\alpha_p\leq \frac{1}{\lambda}$
$R_p^{-1} \gg 1$
and $0 \leq \alpha_p \lesssim 1$,
%{in this} case,  assuming that $\lambda \eqsim 1$,
%$R_p^{-1} \gg 1$ and $0 \leq \alpha_p \lesssim 1$,
{defining the norms according to~\eqref{eq:norm4n}--\eqref{eq:norm6n}}, the boundedness of both
$(\ep (\boldsymbol u), \ep (\boldsymbol w)) +\lambda ( \div \boldsymbol  u, \div \boldsymbol  w)+(R_p^{-1}\bm v, \bm v)$
and $(\div \boldsymbol u,q) +(\div \boldsymbol v,q)$ are again obvious.
Further, {the assumption $0 \leq \alpha_p \lesssim 1$ implies the boundedness of $\alpha_p  (p,q)$.}

Next, {using} the definition of the norms~\eqref{eq:norm4n}--\eqref{eq:norm6n}, Lemma \ref{Stokes:inf-sup},
{and} choosing $(\bm u, \bm v)=(\bm u, \bm 0)\in \bm U\times \bm V$, we {obtain} the inf-sup condition  
\begin{equation}
\inf_{q\in P}\sup_{(\bm u, \bm v)\in \bm U\times \bm V}\frac{(\div \bm u, q)+(\div \bm v, q)}{(\|\bm u\|_{\bm U}+\|\bm v\|_{\boldsymbol V})\|q\|_P}\geq\beta_{su}>0.
\end{equation}
%But the Stokes inf-sup condition is not satisfied. 
{{We see, however, that in this case} the coercivity of
$(\ep (\boldsymbol u), \ep (\boldsymbol u)) +\lambda (\div \bm u, \div \bm u)+(R_p^{-1}\bm v, \bm v)$
can not be valid any more on the kernel set $\bm Z$, where $\bm Z$ is defined by \eqref{kernel}.}
%$$
%\bm Z=\{(\bm u, \bm v)\in \bm U\times\bm V: (\div \bm u, q)+(\div\bm  v, q)=0, \forall q\in P\}.
%$$
In fact, since $(\bm u, \bm v)\in \bm Z$ means $\div \bm v=- \div\bm u$, {it follows that for any $M>0$,} %independent of $R_p^{-1}$, 
there exits $(\bm u,\bm v)$, {where, e.g., $\div \bm v \neq 0$, and $R_p^{-1}$ is large enough,} such that
\begin{equation}\label{non:coercivity}
\begin{split}
&(\ep (\boldsymbol u), \ep (\boldsymbol u)) +(\div \bm u, \div \bm u)+(R_p^{-1}\bm v, \bm v)+R_p^{-1}(\div \bm v, \div\bm v )\\
&\geq M \Big((\ep (\boldsymbol u), \ep (\boldsymbol u)) +(\div \bm u, \div \bm u)+(R_p^{-1}\bm v, \bm v)\Big)\\
& {\gtrsim} \, M \Big((\ep (\boldsymbol u), \ep (\boldsymbol u)) +\lambda(\div \bm u, \div \bm u)+(R_p^{-1}\bm v, \bm v)\Big),
\end{split}
\end{equation}
where the second inequality comes from {$\lambda \lesssim 1$}.

{Using the norms defined in~\eqref{eq:norm4n}--\eqref{eq:norm5n}},
{the estimate}~\eqref{non:coercivity} implies that for
any $M>0$, there exists $(\bm u,\bm v)\in \bm Z$ (and $R_p^{-1}$  large enough) such that
\begin{equation}
\|\bm u\|^2_{\bm U}+\|\bm v\|^2_{\bm V}\geq M \Big((\ep (\boldsymbol u), \ep (\boldsymbol u)) +(\div \bm u, \div \bm u)+(R_p^{-1}\bm v, \bm v\Big).
\end{equation}
Therefore the system  \eqref{eq:8a}--\eqref{eq:8c} is not {uniformly} stable with respect to the parameter $R_p^{-1}$
under the norms \eqref{eq:norm4n}--\eqref{eq:norm6n}. 
%In this way, we showed that $P_1/RT_0/P_0$ are not stable for $R_p^{-1}$.

{From this observation we conclude that we have to define proper norms (as we do below in~\eqref{eq:7a}--\eqref{eq:7c})
in order to establish the coercivity of
$(\ep (\boldsymbol u), \ep (\boldsymbol u))+\lambda (\div \bm u, \div \bm u)+(R_p^{-1}\bm v, \bm v)$ on $\bm Z$ {in} both
above cases.}
%%Hence we arrive in the following Section \ref{sec:par_rob_stab_model}.

%But as we have point out in remark \ref{Stokes_necessary}, the $H(\div)$ inf-sup condition is not satisfied any more. Hence in order to get
%the big inf-sup condition, we need to choose another space such that the Stokes inf-sup condition is satisfied. In our paper, we used $H(\div)$ conforming spaces such as $BDM, RT$ and $BDFM$ to replace $P_1$.  

\section{Parameter-robust stability of the model}\label{sec:par_rob_stab_model}

In this section, we first define proper parameter-dependent norms for the spaces $\bm U, \bm V$ and $P$
based on which we establish then the parameter-robust stability of the Biot's model~\eqref{eq:8a}--\eqref{eq:8c}
for parameters in the ranges
\begin{equation}\label{parameter:range}
%1\leq\lambda\leq \infty, 0\leq R^{-1}_p\leq \infty, 0\leq \alpha_p \leq \infty.
{\lambda \ge 1, \quad  R^{-1}_p \ge 0, \quad \alpha_p \ge 0.}
\end{equation}
%%{We should note that range of the parameters are very general and reasonable.}

{Let us denote}
\begin{equation}\label{eq:6}
\rho=\min\{\lambda, R^{-1}_p\},~~ \gamma=\max\{\rho^{-1}, \alpha_p\},
\end{equation}
and consider the Hilbert spaces $\bU=H^1_0(\Omega)^d, \bV=H_0(\div, \Omega), P= L^2_0(\Omega)$
with parameter-dependent norms $\Vert \cdot \Vert_{\bU}$, $\Vert \cdot \Vert_{\bV}$,
$\Vert \cdot \Vert_P$ induced by the \emph{inner products}
\begin{subequations}\label{eq:7}
\begin{eqnarray}
(\boldsymbol u, \boldsymbol w)_{\bU}&=&(\ep (\boldsymbol u), \ep (\boldsymbol w))+\lambda(\div \boldsymbol u, \div \boldsymbol w),\label{eq:7a}\\
(\boldsymbol v, \boldsymbol z)_{\bV}&=&(R^{-1}_p \boldsymbol v,\boldsymbol z)+\gamma^{-1} (\div \boldsymbol v, \div \boldsymbol z),\label{eq:7b}\\
(p, q)_P&=&\gamma (p,q).\label{eq:7c}
\end{eqnarray}
\end{subequations}
%These parameter dependent norms 
The above norms are the key to establish the parameter-robust stability of the model.
%

%For simplicity, we still write the weak formulation of  \eqref{eq:5a}-\eqref{eq:5c} which repeats  \eqref{eq:8a}--\eqref{eq:8c}  as: Find $(\boldsymbol u, \boldsymbol v, p)\in \boldsymbol U\times\boldsymbol V\times P$, such that for any
%$(\boldsymbol w, \boldsymbol z, q)\in \boldsymbol U\times\boldsymbol V\times P$
%\begin{subequations}\label{eq:8}
%\begin{eqnarray}
%(\ep (\boldsymbol u), \ep (\boldsymbol w)) +\lambda ( \div \boldsymbol  u, \div \boldsymbol  w)- ( p, \div \boldsymbol  w)&=&(\boldsymbol f, \boldsymbol w),\label{eq:8a}\\
%(R_p^{-1}\boldsymbol  v, \boldsymbol  z)- (p, \div \boldsymbol  z)&=&\boldsymbol 0, \label{eq:8b} \\
%-(\div \boldsymbol u,q)  -(\div \boldsymbol v,q)  -\alpha_p  (p,q)&=&(g,q). \label{eq:8c}
%\end{eqnarray}
%\end{subequations}

Directly related to problem~\eqref{eq:8a}--\eqref{eq:8c} we introduce the bilinear form 
\begin{equation}\label{eq:42}
\begin{split}
A((\boldsymbol u,  \boldsymbol v, p), (\boldsymbol w,  \boldsymbol z, q))=&(\ep (\boldsymbol u), \ep (\boldsymbol w)) +\lambda ( \div \boldsymbol  u, \div \boldsymbol  w)- ( p, \div \boldsymbol  w)+
  (R_p^{-1}\boldsymbol  v, \boldsymbol  z)- (p, \div \boldsymbol  z)\\
-&(\div \boldsymbol u,q)  -(\div \boldsymbol v,q)  -\alpha_p  (p,q).
\end{split}
\end{equation}
{{In view of} the definition of the norms \eqref{eq:7a}--\eqref{eq:7c},  the boundedness of
$A((\boldsymbol u,  \boldsymbol v, p), (\boldsymbol w,  \boldsymbol z, q))$ is obvious.
{We come to our first main result.}}
\begin{theorem}\label{Continuity:Stability}
There exists a constant $\beta>0$ independent of  the parameters $\lambda,~R_p^{-1},~\alpha_p$, such that 
\begin{equation}\label{eq:43}
\inf_{(\boldsymbol u,  \boldsymbol v, p)\in \boldsymbol U\times\boldsymbol V\times P} \sup_{(\boldsymbol w,\boldsymbol z, q)\in \boldsymbol U\times\boldsymbol V\times P}\frac{A((\boldsymbol u,  \boldsymbol v, p), (\boldsymbol w,  \boldsymbol z, q))}{(\|\boldsymbol u\|_{\boldsymbol U}+\|\boldsymbol v\|_{\boldsymbol V}+\|p\|_P)(\|\boldsymbol w\|_{\boldsymbol U}+\|\boldsymbol z\|_{\boldsymbol V}+\|q\|_P)}\geq \beta.
\end{equation}
\end{theorem}
\begin{proof}
\underline{Case~\uppercase\expandafter{\romannumeral 1}:} 
\begin{equation}\label{eq:44}
\rho=\min\{\lambda,R^ {-1}_p\}=\lambda, ~\hbox{hence} ~\lambda\leq R_p^{-1}, ~\gamma^{-1}\leq \rho=\lambda.
\end{equation}
For any $(\boldsymbol u,  \boldsymbol v, p)\in \boldsymbol U\times\boldsymbol V\times P$, by Lemma~\ref{Stokes:inf-sup},
there exists 
\begin{equation}\label{eq:45}
\boldsymbol u_0\in \boldsymbol U,~\hbox{such that}~\div \boldsymbol u_0=\frac{1}{\sqrt \lambda} p,~ \|\boldsymbol u_0\|_1\leq \beta_s^{-1}\frac{1}{\sqrt \lambda} \|p\|.
\end{equation}
Choose 
\begin{equation}\label{eq:wzq}
\boldsymbol w=\delta \boldsymbol u-\frac{1}{\sqrt \lambda} \boldsymbol u_0,\;\;\; \boldsymbol z=\delta \boldsymbol v, \;\;\;q= -\delta p-\gamma^{-1}\div \boldsymbol v,
\end{equation}
where $\delta$ is a positive constant that will be determined later.

First we verify the boundedness of $(\boldsymbol w,  \boldsymbol z, q)$ by $(\boldsymbol u,  \boldsymbol v, p)$.

By \eqref{eq:45}, and noting that $\gamma^{-1}\leq \rho=\lambda$, we have 
\begin{equation}\label{eq:46}
\begin{split}
(\frac{1}{\sqrt \lambda} \boldsymbol u_0, \frac{1}{\sqrt \lambda} \boldsymbol u_0)_{\boldsymbol U}&=(\ep (\frac{1}{\sqrt \lambda}\boldsymbol u_0), \ep (\frac{1}{\sqrt \lambda}\boldsymbol u_0))+\lambda (\div \frac{1}{\sqrt \lambda} \boldsymbol u_0, \div \frac{1}{\sqrt \lambda} \boldsymbol u_0)\\
&=(\frac{1}{\lambda} \ep(\boldsymbol u_0),  \ep(\boldsymbol u_0))+(\div \boldsymbol u_0, \div \boldsymbol u_0)\leq \frac{1}{\lambda} \beta^{-2}_s \frac{1}{\lambda}\|p\|^2+
\frac{1}{\lambda}(p,p)\\
&\leq (\frac{1}{\lambda} \beta^{-2}_s+1)\frac{1}{\lambda} \|p\|^2\leq (\frac{1}{\lambda} \beta^{-2}_s+1) \gamma \|p\|^2=(\frac{1}{\lambda} \beta^{-2}_s+1) (p,p)_P.
\end{split}
\end{equation}
Next, since $1\leq \lambda$, we get the boundedness of $\boldsymbol w$, i.e.,
\begin{equation}\label{eq:47}
\|\boldsymbol w\|_{\boldsymbol U}\leq \delta \|\boldsymbol u\|_{\boldsymbol U}+\sqrt{\lambda ^{-1}\beta^{-2}_s+1} \|p\|_P\leq  \delta \|\boldsymbol u\|_{\boldsymbol U}+\sqrt{\beta^{-2}_s+1} \|p\|_P.
\end{equation}
It is obvious that $\|\boldsymbol z\|_{\boldsymbol V}= \delta\|\boldsymbol v\|_{\boldsymbol V}$.
We still need to bound $q$.
{Using \eqref{eq:wzq} and}
\begin{equation}\label{eq:48}
(\gamma^{-1} \div \boldsymbol v, \gamma^{-1} \div \boldsymbol v)_P=\gamma (\gamma^{-1} \div \boldsymbol v, \gamma^{-1} \div \boldsymbol v)=\gamma^{-1} (\div \boldsymbol v, \div \boldsymbol v)\leq (\boldsymbol v, \boldsymbol  v)_{\boldsymbol V}
\end{equation}
we obtain
\begin{equation}\label{eq:49}
\|q\|_P\leq \delta \|p\|_P+\|\boldsymbol  v\|_{\boldsymbol V}.
\end{equation}
{Collecting the estimates above we get
\begin{equation}\label{eq:boundedness}
\|\boldsymbol w\|_{\boldsymbol U}+\|\boldsymbol z\|_{\boldsymbol V}+\|q\|_P \leq
%\lesssim \frac{\delta}{\beta_s \sqrt \lambda} \|(\boldsymbol u,  \boldsymbol v, p)\|_{\boldsymbol W} .
\left( \delta + \sqrt{\beta_s^{-2}+1}\right) 
\left(
\|\boldsymbol u\|_{\boldsymbol U}+\|\boldsymbol v\|_{\boldsymbol V}+\|p\|_P \right).
%=: C_0 \left(
%\|\boldsymbol u\|_{\boldsymbol U}+\|\boldsymbol v\|_{\boldsymbol V}+\|p\|_P \right).
\end{equation}}

Next we show the coercivity of $A((\boldsymbol u,  \boldsymbol v, p), (\boldsymbol w,  \boldsymbol z, q))$.
Using the definition of $(\boldsymbol w,  \boldsymbol z, q)$ and \eqref{eq:45}, we find
\begin{eqnarray*}
A((\boldsymbol u,  \boldsymbol v, p), (\boldsymbol w,  \boldsymbol z, q))&=&(\ep (\boldsymbol u), \ep (\boldsymbol w)) +\lambda ( \div \boldsymbol  u, \div \boldsymbol  w)- ( p, \div \boldsymbol  w)+
  (R_p^{-1}\boldsymbol  v, \boldsymbol  z)- (p, \div \boldsymbol  z)\\
&&-(\div \boldsymbol u,q)  -(\div \boldsymbol v,q)  -\alpha_p  (p,q)\\
&=&(\ep (\boldsymbol u), \delta \ep (\boldsymbol u)-\frac{1}{\sqrt \lambda}\ep (\boldsymbol u_0)) +\lambda ( \div \boldsymbol  u, \delta\div \boldsymbol  u-\frac{1}{\sqrt \lambda}\div \boldsymbol u_0 )- ( p, \delta\div \boldsymbol  u-\frac{1}{\sqrt \lambda}\div \boldsymbol u_0)\\
&&+(R_p^{-1}\boldsymbol  v, \delta\boldsymbol  v)- (p,  \delta\div \boldsymbol  v)-(\div \boldsymbol u, -\delta p-\gamma^{-1}\div \boldsymbol v)\\
&&-(\div \boldsymbol v,-\delta p-\gamma^{-1}\div \boldsymbol v)  -\alpha_p  (p,-\delta p-\gamma^{-1}\div \boldsymbol v)\\
&=&\delta(\ep (\boldsymbol u), \ep (\boldsymbol u))-\frac{1}{\sqrt \lambda}(\ep (\boldsymbol u), \ep (\boldsymbol u_0)) +\delta \lambda ( \div \boldsymbol  u, \div \boldsymbol  u)-{\sqrt \lambda}(\div \boldsymbol  u,\div \boldsymbol u_0 )\\
&&- \delta(p,\div \boldsymbol  u)+\frac{1}{\sqrt \lambda}(p,\div \boldsymbol u_0)
+\delta(R_p^{-1}\boldsymbol  v, \boldsymbol  v)- \delta(p,  \div \boldsymbol  v)+\delta(\div \boldsymbol u, p)\\
&&+\gamma^{-1}(\div \boldsymbol u, \div \boldsymbol v)+\delta(\div \boldsymbol v, p)+\gamma^{-1}(\div \boldsymbol v,\div \boldsymbol v)  +\delta \alpha_p  (p, p) + \alpha_p(p, \gamma^{-1}\div \boldsymbol v)\\
&=&\delta(\ep (\boldsymbol u), \ep (\boldsymbol u))-\frac{1}{\sqrt \lambda}(\ep (\boldsymbol u), \ep (\boldsymbol u_0)) +\delta \lambda ( \div \boldsymbol  u, \div \boldsymbol  u)-(\div \boldsymbol  u, p)+\frac{1}{\lambda}(p, p)~~~(\hbox{by}~\eqref{eq:45})\\
&&+\delta(R_p^{-1}\boldsymbol  v, \boldsymbol  v)+\gamma^{-1}(\div \boldsymbol u, \div \boldsymbol v)+\gamma^{-1}(\div \boldsymbol v,\div \boldsymbol v)  +\delta \alpha_p  (p, p) + \alpha_p(p, \gamma^{-1}\div \boldsymbol v) .
\end{eqnarray*}
Applying Cauchy's inequality and using~\eqref{eq:45}, we therefore obtain
\begin{eqnarray*}
A((\boldsymbol u,  \boldsymbol v, p), (\boldsymbol w,  \boldsymbol z, q))
&=&\delta(\ep (\boldsymbol u), \ep (\boldsymbol u))-\frac{1}{\sqrt \lambda}(\ep (\boldsymbol u), \ep (\boldsymbol u_0)) +\delta \lambda ( \div \boldsymbol  u, \div \boldsymbol  u)-(\div \boldsymbol  u, p)+\frac{1}{\lambda}(p, p)\\
&&+\delta(R_p^{-1}\boldsymbol  v, \boldsymbol  v)+\gamma^{-1}(\div \boldsymbol u, \div \boldsymbol v)+\gamma^{-1}(\div \boldsymbol v,\div \boldsymbol v)  +\delta \alpha_p  (p, p) + \alpha_p(p, \gamma^{-1}\div \boldsymbol v)\\
&\geq&\delta(\ep (\boldsymbol u), \ep (\boldsymbol u))-\frac{1}{2}\frac{1}{\sqrt \lambda}\epsilon_1^{-1}(\ep (\boldsymbol u), \ep (\boldsymbol u))-\frac{1}{2}\frac{1}{\sqrt \lambda}\epsilon_1(\ep (\boldsymbol u_0), \ep (\boldsymbol u_0)) +\delta \lambda ( \div \boldsymbol  u, \div \boldsymbol  u)\\
&&-\frac{1}{2}\epsilon_2^{-1}\lambda(\div \boldsymbol  u, \div \boldsymbol  u)-\frac{1}{2}\epsilon_2\frac{1}{\lambda}(p, p)+\frac{1}{\lambda}(p, p)+\delta(R_p^{-1}\boldsymbol  v, \boldsymbol  v)-\frac{1}{2}\epsilon_3^{-1}\gamma^{-1}(\div \boldsymbol u, \div \boldsymbol u)\\
&&-\frac{1}{2}\epsilon_3\gamma^{-1}(\div \boldsymbol v, \div \boldsymbol v)+\gamma^{-1}(\div \boldsymbol v,\div \boldsymbol v)  +\delta \alpha_p  (p, p) -\frac{1}{2}\epsilon_4\gamma^{-1}(\div \boldsymbol v, \div \boldsymbol v)\\
&& -\frac{1}{2}\epsilon_4^{-1}\alpha_p^2\gamma^{-1}(p, p)\\
&\geq&(\delta-\frac{1}{2}\frac{1}{\sqrt \lambda}\epsilon_1^{-1})(\ep (\boldsymbol u), \ep (\boldsymbol u)) +(\delta -\frac{1}{2}\epsilon_2^{-1})\lambda ( \div \boldsymbol  u, \div \boldsymbol  u)-\frac{1}{2}\epsilon_3^{-1}\gamma^{-1}(\div \boldsymbol u, \div \boldsymbol u)\\
&&+\delta(R_p^{-1}\boldsymbol  v, \boldsymbol  v)+(1-\frac{1}{2}\epsilon_3-\frac{1}{2}\epsilon_4)\gamma^{-1}(\div \boldsymbol v,\div \boldsymbol v) \\
&&+(1-\frac{1}{2}\frac{1}{\sqrt \lambda}\epsilon_1\beta_s^{-2}-\frac{1}{2}\epsilon_2)\frac{1}{\lambda}(p, p)+(\delta -\frac{1}{2}\epsilon_4^{-1}\alpha_p\gamma^{-1})\alpha_p  (p, p)  ~~~~~(\hbox{by}~\eqref{eq:45}) .
\end{eqnarray*}
%
%Letting} $\epsilon_1=\displaystyle\frac{\beta_s^2}{4}, \epsilon_2=\epsilon_3=\epsilon_4=\displaystyle\frac{1}{4}$ and noting that $\rho^{-1}\leq \gamma, \lambda\geq \gamma^{-1}>0$, we further conclude that
%
%\begin{eqnarray*}
%A((\boldsymbol u,  \boldsymbol v, p), (\boldsymbol w,  \boldsymbol z, q))
%&\geq&(\delta-\frac{1}{\sqrt \lambda}4 \beta_s^{-2})(\ep (\boldsymbol u), \ep (\boldsymbol u)) +(\delta -8)\lambda ( \div \boldsymbol  u, \div \boldsymbol  u)+\delta(R_p^{-1}\boldsymbol  v, \boldsymbol  v)\\
%&&+\frac{1}{2}\gamma^{-1}(\div \boldsymbol v,\div \boldsymbol v)+\frac{1}{2}\frac{1}{\lambda}(p, p)+(\delta -4\alpha_p\gamma^{-1})\alpha_p  (p, p)
%\end{eqnarray*}
%Now, let $\delta=\max\{4 \beta_s^{-2}+\displaystyle\frac{1}{2}, 8+\displaystyle\frac{1}{2}\}$. Then by noting that $\alpha_p\leq \gamma$, we arrive at the following coervicity estimate
%\begin{eqnarray}\label{eq:60}
%A((\boldsymbol u,  \boldsymbol v, p), (\boldsymbol w,  \boldsymbol z, q))
%&\geq&\frac{1}{2}(\ep (\boldsymbol u), \ep (\boldsymbol u)) +\frac{1}{2}\lambda ( \div \boldsymbol  u, \div \boldsymbol  u)+\frac{1}{2}(R_p^{-1}\boldsymbol  v, \boldsymbol  v) \nonumber \\
%&&+\frac{1}{2}\gamma^{-1}(\div \boldsymbol v,\div \boldsymbol v)+\frac{1}{2}\frac{1}{\lambda}(p, p)+\frac{1}{2}\alpha_p  (p, p) \\
%&\geq&\frac{1}{2}\big(\|\boldsymbol u\|^2_{\boldsymbol U}+\|\boldsymbol  v\|^2_{\boldsymbol V}+\|p\|^2_P\big). \nonumber
%\end{eqnarray}

{Now, letting $\epsilon_1=\frac{\beta_s^2}{2}, \epsilon_2=\epsilon_3=\epsilon_4=\frac{1}{2}$ and noting that
$\rho^{-1}\leq \gamma, \lambda\geq \gamma^{-1}>0$ and $\lambda\geq 1$, we further conclude that
\begin{eqnarray*}
A((\boldsymbol u,  \boldsymbol v, p), (\boldsymbol w,  \boldsymbol z, q))
&\geq&(\delta-\frac{1}{\sqrt \lambda} \beta_s^{-2})(\ep (\boldsymbol u), \ep (\boldsymbol u)) +(\delta -2)\lambda ( \div \boldsymbol  u, \div \boldsymbol  u)+\delta(R_p^{-1}\boldsymbol  v, \boldsymbol  v)\\
&&+\frac{1}{2}\gamma^{-1}(\div \boldsymbol v,\div \boldsymbol v)+\frac{1}{2}\frac{1}{\lambda}(p, p)+(\delta -\alpha_p\gamma^{-1})\alpha_p  (p, p) .
\end{eqnarray*}
Next, {letting} $\delta:=\max\{\beta_s^{-2}+\frac{1}{2}, 2+\frac{1}{2}\}$ {and noting} that $\alpha_p\leq \gamma, \lambda\geq 1$, we arrive at the following coervicity estimate
\begin{eqnarray}\label{eq:60}
A((\boldsymbol u,  \boldsymbol v, p), (\boldsymbol w,  \boldsymbol z, q))
&\geq&\frac{1}{2}(\ep (\boldsymbol u), \ep (\boldsymbol u)) +\frac{1}{2}\lambda ( \div \boldsymbol  u, \div \boldsymbol  u)+\frac{1}{2}(R_p^{-1}\boldsymbol  v, \boldsymbol  v) \nonumber \\
&&+\frac{1}{2}\gamma^{-1}(\div \boldsymbol v,\div \boldsymbol v)+\frac{1}{2}\frac{1}{\lambda}(p, p)+\alpha_p  (p, p) \\
& \ge& \frac{1}{2} \big(\|\boldsymbol u\|_{\boldsymbol U}+\|\boldsymbol  v\|_{\boldsymbol V}+\|p\|_P\big)^2. \nonumber
%\\
%&\geq&\big(\|\boldsymbol u\|^2_{\boldsymbol U}+\|\boldsymbol  v\|^2_{\boldsymbol V}+\|p\|^2_P\big). \nonumber
\end{eqnarray}}
%which holds for all $\epsilon>0$ and therefore
%$$
%A((\boldsymbol u,  \boldsymbol v, p), (\boldsymbol w,  \boldsymbol z, q))
%%&\geq&(\ep (\boldsymbol u), \ep (\boldsymbol u)) +\lambda ( \div \boldsymbol  u, \div \boldsymbol  u)+(R_p^{-1}\boldsymbol  v, \boldsymbol  v) \nonumber \\
%%&&+(1-2\epsilon)\gamma^{-1}(\div \boldsymbol v,\div \boldsymbol v)+(1-\epsilon)\frac{1}{\lambda}(p, p)+\alpha_p  (p, p) \\
%\geq \big(\|\boldsymbol u\|^2_{\boldsymbol U}+\|\boldsymbol  v\|^2_{\boldsymbol V}+\|p\|^2_P\big)
%\ge \frac{1}{3}\big(\|\boldsymbol u\|_{\boldsymbol U}+\|\boldsymbol  v\|_{\boldsymbol V}+\|p\|_P\big)^2.
%$$
%{
%Since we can choose $\epsilon > 0$ such that the estimates~\eqref{eq:boundedness} and~\eqref{eq:60} both hold
%for some positive constants $C_0$ and $C_1$ independent of the parameters $\lambda,~R_p^{-1},~\alpha_p$ this
%proves the statement of the theorem in Case~\uppercase\expandafter{\romannumeral 1}.}
%\smallskip

\underline{Case~\uppercase\expandafter{\romannumeral 2}:}
\begin{equation}\label{eq:61}
\rho=\min\{\lambda,R^ {-1}_p\}=R^ {-1}_p, ~\hbox{hence} ~\lambda\geq R_p^{-1}, ~\gamma^{-1}\leq \rho=R^ {-1}_p.
\end{equation}
For any $(\boldsymbol u,  \boldsymbol v, p)\in \boldsymbol U\times\boldsymbol V\times P$, by Lemma \ref{Hdiv:inf-sup}, there exists 
\begin{equation}\label{eq:62}
\boldsymbol v_0\in \boldsymbol V,~\hbox{such that}~\div \boldsymbol v_0={\sqrt {R_p}} p,~ \|\boldsymbol v_0\|_{\div}\leq \beta_d^{-1}{\sqrt {R_p}} \|p\|.
\end{equation}
Choose 
\begin{equation}\label{eq:63}
\boldsymbol w=\delta \boldsymbol u,\;\;\; \boldsymbol z=\delta \boldsymbol v-\sqrt {R_p}\boldsymbol v_0,\;\;\; q= -\delta p-\gamma^{-1}\div \boldsymbol v,
\end{equation}
where $\delta$ is a constant which we will specify later.

Again we verify the boundedness of $(\boldsymbol w,  \boldsymbol z, q)$ by $(\boldsymbol u,  \boldsymbol v, p)$ first.
It is obvious that $\|\boldsymbol w\|_{\boldsymbol U}=\delta \|\boldsymbol u\|_{\boldsymbol U}$.

{Moreover,} by \eqref{eq:63} and noting that $\gamma^{-1}\leq \rho=R_p^{-1}$, we have 
\begin{equation}\label{eq:64}
\begin{split}
({\sqrt R_p} \boldsymbol v_0, {\sqrt R_p} \boldsymbol v_0)_{\boldsymbol V}&=(R_p^{-1} {\sqrt R_p}\boldsymbol v_0), {\sqrt R_p}\boldsymbol v_0)+\gamma^{-1}(\div\sqrt {R_p} \boldsymbol v_0, \div\sqrt {R_p} \boldsymbol v_0)\\
&\leq(\boldsymbol v_0,  \boldsymbol v_0)+(\div \boldsymbol v_0, \div \boldsymbol v_0)\leq \beta^{-2}_d R_p\|p\|^2\leq \beta^{-2}_d \gamma\|p\|^2\leq\beta^{-2}_d(p,p)_P.
\end{split}
\end{equation}
Hence, we get the boundedness of $\boldsymbol z$, {i.e.},
\begin{equation}\label{eq:65}
\|\boldsymbol z\|_{\boldsymbol V}\leq \delta \|\boldsymbol v\|_{\boldsymbol V}+\beta^{-1}_d \|p\|_P.
\end{equation}
{The boundedness of $q$ follows as in~\eqref{eq:49}.}

Next we verify the coercivity of $A((\boldsymbol u,  \boldsymbol v, p), (\boldsymbol w,  \boldsymbol z, q))$
in Case~\uppercase\expandafter{\romannumeral 2}. 

{Using} the definition of $(\boldsymbol w,  \boldsymbol z, q)$ and~\eqref{eq:62}, we find 
\begin{eqnarray*}
A((\boldsymbol u,  \boldsymbol v, p), (\boldsymbol w,  \boldsymbol z, q))&=&(\ep (\boldsymbol u), \ep (\boldsymbol w)) +\lambda ( \div \boldsymbol  u, \div \boldsymbol  w)- ( p, \div \boldsymbol  w)+
  (R_p^{-1}\boldsymbol  v, \boldsymbol  z)- (p, \div \boldsymbol  z)\\
&&-(\div \boldsymbol u,q)  -(\div \boldsymbol v,q)  -\alpha_p  (p,q)\\
&=&(\ep (\boldsymbol u), \delta \ep (\boldsymbol u)) +\lambda ( \div \boldsymbol  u, \delta\div \boldsymbol  u )- ( p, \delta\div \boldsymbol  u)+(R_p^{-1}\boldsymbol  v, \delta\boldsymbol  v-\sqrt{R_p} \boldsymbol v_0)\\
&&-(p,  \delta\div \boldsymbol  v-\sqrt{R_p} \div\boldsymbol v_0)-(\div \boldsymbol u, -\delta p-\gamma^{-1}\div \boldsymbol v)\\
&&-(\div \boldsymbol v,-\delta p-\gamma^{-1}\div \boldsymbol v)  -\alpha_p  (p,-\delta p-\gamma^{-1}\div \boldsymbol v)\\
&=&\delta(\ep (\boldsymbol u), \ep (\boldsymbol u))+\delta \lambda ( \div \boldsymbol  u, \div \boldsymbol  u)-\delta(p,\div \boldsymbol  u)+\delta(R_p^{-1}\boldsymbol  v, \boldsymbol  v)-({R^{-1/2}_p}\boldsymbol  v, \boldsymbol v_0)\\
&&-\delta(p,  \div \boldsymbol  v)+(p, \sqrt{R_p} \div\boldsymbol v_0)+\delta(\div \boldsymbol u, p)+\gamma^{-1}(\div \boldsymbol u, \div \boldsymbol v)\\
&&+\delta(\div \boldsymbol v, p)+\gamma^{-1}(\div \boldsymbol v,\div \boldsymbol v)  +\delta \alpha_p  (p, p) + \alpha_p(p, \gamma^{-1}\div \boldsymbol v)\\
&=&\delta(\ep (\boldsymbol u), \ep (\boldsymbol u))+\delta \lambda ( \div \boldsymbol  u, \div \boldsymbol  u)+\delta(R_p^{-1}\boldsymbol  v, \boldsymbol  v)-({R^{-1/2}_p}\boldsymbol  v, \boldsymbol v_0)\\
&&+(p, \sqrt{R_p} \div\boldsymbol v_0)+\gamma^{-1}(\div \boldsymbol u, \div \boldsymbol v)\\
&&+\gamma^{-1}(\div \boldsymbol v,\div \boldsymbol v)  +\delta \alpha_p  (p, p) + \alpha_p(p, \gamma^{-1}\div \boldsymbol v)\\
\end{eqnarray*}
{Applying Cauchy's inequality and using~\eqref{eq:62}, we get 
\begin{eqnarray*}
&&A((\boldsymbol u,  \boldsymbol v, p), (\boldsymbol w,  \boldsymbol z, q))\\
&=&\delta(\ep (\boldsymbol u), \ep (\boldsymbol u))+\delta \lambda ( \div \boldsymbol  u, \div \boldsymbol  u)+\delta(R_p^{-1}\boldsymbol  v, \boldsymbol  v)-({R^{-1/2}_p}\boldsymbol  v, \boldsymbol v_0)\\
&&+(p, \sqrt{R_p} \div\boldsymbol v_0)+\gamma^{-1}(\div \boldsymbol u, \div \boldsymbol v)\\
&&+\gamma^{-1}(\div \boldsymbol v,\div \boldsymbol v)  +\delta \alpha_p  (p, p) + \alpha_p(p, \gamma^{-1}\div \boldsymbol v)\\
&\geq&\delta(\ep (\boldsymbol u), \ep (\boldsymbol u))+\delta \lambda ( \div \boldsymbol  u, \div \boldsymbol  u)+\delta(R_p^{-1}\boldsymbol  v, \boldsymbol  v)-\frac{1}{2}\epsilon_1^{-1}({R^{-1}_p}\boldsymbol  v, \boldsymbol v)-\frac{1}{2}\epsilon_1(\boldsymbol  v_0, \boldsymbol v_0)\\
&&+({R_p} p, p)-\frac{1}{2}\epsilon_2^{-1}\gamma^{-1}(\div \boldsymbol u, \div \boldsymbol u)-\frac{1}{2}\epsilon_2\gamma^{-1}(\div \boldsymbol v, \div \boldsymbol v)\\
&&+\gamma^{-1}(\div \boldsymbol v,\div \boldsymbol v)  +\delta \alpha_p  (p, p) -\frac{1}{2}\epsilon_4\gamma^{-1}(\div \boldsymbol v, \div \boldsymbol v) -\frac{1}{2}\epsilon_4^{-1}\alpha_p^2\gamma^{-1}(p, p)\\
&\geq&\delta(\ep (\boldsymbol u), \ep (\boldsymbol u))+\delta \lambda ( \div \boldsymbol  u, \div \boldsymbol  u)-\frac{1}{2}\epsilon_2^{-1}\gamma^{-1}(\div \boldsymbol u, \div \boldsymbol u)+(\delta-\frac{1}{2}\epsilon_1^{-1})({R^{-1}_p}\boldsymbol  v, \boldsymbol v)\\
&&+(1-\frac{1}{2}\epsilon_2-\frac{1}{2}\epsilon_4)\gamma^{-1}(\div \boldsymbol v, \div \boldsymbol v) +(1-\frac{1}{2}\epsilon_1\beta^{-2}_d)(R_p p, p)+(\delta -\frac{1}{2}\epsilon_4^{-1}\alpha_p\gamma^{-1})\alpha_p  (p, p)\\
\end{eqnarray*}
Now, letting $\epsilon_1=\displaystyle\beta_d^2, \epsilon_2=\epsilon_4=\displaystyle\frac{1}{2}$
and noting that $\rho^{-1}\leq \gamma, \lambda\geq \rho\geq \gamma^{-1}>0$ it follows that
\begin{eqnarray*}
A((\boldsymbol u,  \boldsymbol v, p), (\boldsymbol w,  \boldsymbol z, q))
&\geq&\delta(\ep (\boldsymbol u), \ep (\boldsymbol u)) +(\delta-1)\lambda (\div \boldsymbol  u, \div \boldsymbol  u)+(\delta-\frac{1}{2}\beta_d^{-2})(R_p^{-1}\boldsymbol  v, \boldsymbol  v)\\
&&+\frac{1}{2}\gamma^{-1}(\div \boldsymbol v,\div \boldsymbol v)+\frac{1}{2}(R_pp, p)+(\delta -\alpha_p\gamma^{-1})\alpha_p  (p, p)
\end{eqnarray*}
Next, we set $\delta=\max\left\{\frac{1}{2}\beta_d^{-2}+\displaystyle\frac{1}{2}, 1+\displaystyle\frac{1}{2}\right\}$,}
observe that  $\alpha_p\leq \gamma$, and finally obtain the coercivity estimate
\begin{eqnarray*}
A((\boldsymbol u,  \boldsymbol v, p), (\boldsymbol w,  \boldsymbol z, q)) 
&\geq&\frac{1}{2}(\ep (\boldsymbol u), \ep (\boldsymbol u)) +\frac{1}{2}\lambda (\div \boldsymbol  u, \div \boldsymbol  u)+\frac{1}{2}(R_p^{-1}\boldsymbol  v, \boldsymbol  v) \\
&&+\frac{1}{2}\gamma^{-1}(\div \boldsymbol v,\div \boldsymbol v)+\frac{1}{2}(R_pp, p)+\frac{1}{2}\alpha_p  (p, p) \\
&\geq&\frac{1}{2}\big(\|\boldsymbol u\|^2_{\boldsymbol U}+\|\boldsymbol  v\|^2_{\boldsymbol V}+\|p\|^2_P\big)
\end{eqnarray*}
which completes the proof.
\end{proof}
The above theorem implies the following stability estimate.
\begin{corollary}\label{eq:67}
Let $(\boldsymbol u,  \boldsymbol v, p)\in \boldsymbol U\times \boldsymbol V\times P$ be the solution
of~\eqref{eq:8a}--\eqref{eq:8c}. Then {there holds} the estimate
\begin{equation}
\|\boldsymbol u\|_{\boldsymbol U}+\|\boldsymbol  v\|_{\boldsymbol V}+\|p\|_P\leq C (\|\boldsymbol f\|_{\boldsymbol U^*}+\|g\|_{P^*}),
\end{equation} 
where $C$ is a constant independent of $\lambda, R_p^{-1}, \alpha_p$ and $\|\boldsymbol f\|_{\boldsymbol U^*}=\sup\limits_{\boldsymbol w\in \boldsymbol U}\frac{(\boldsymbol f, \boldsymbol w)}{\|\boldsymbol w\|_{\boldsymbol U}},\|g\|_{P^*}=\sup\limits_{q\in P}\frac{(g,q)}{\|q\|_P}=\gamma^{-1}\|g\|$.
\end{corollary}
\begin{remark}
We want to emphasize that the parameter ranges as specified in~\eqref{parameter:range}
are indeed relevant since the variations of the model parameters are quite large in many applications.
For that reason, Theorem~\ref{Continuity:Stability} is a very important and basic result that provides the
parameter-robust stability of the model~\eqref{eq:8a}--\eqref{eq:8c}.
\end{remark}
{
\begin{remark}
Define 
\begin{equation}\label{Preconditioner:B}
B:=\left[\begin{array}{ccc}
(-\div \ep-\lambda \nabla \div)^{-1} & 0 & 0 \\
0 & (R_p^{-1}I+ \gamma^{-1}\nabla \div)^{-1}& 0 \\
0& 0& (\gamma I )^{-1}  \end{array}\right].
\end{equation}
{Due to} the theory presented in~\cite{Mardal2011preconditioning},
Theorem~\ref{Continuity:Stability} implies that the 
{operator $B$ in~\eqref{Preconditioner:B} defines a norm-equivalent
(canonical) block-diagonal preconditioner for the operator $A$ in~\eqref{operator:A}
which is robust in all model parameters.}
\end{remark}
}
%\section{$P_1\times RT_0\times P_0$ is not uniform stable for $R_p^{-1}$}\label{sec:com_unstable_disc}
\begin{remark}\label{Stokes_necessary}
Note that if $\lambda \ll R_p^{-1}$ and $0\leq\alpha_p\leq \frac{1}{\lambda}$
then $\rho=\min\{\lambda,R^ {-1}_p\}=\lambda$ and the norms defined in \eqref{eq:7a}--\eqref{eq:7c}
{are given by}
\begin{eqnarray}
(\boldsymbol u, \boldsymbol w)_{\bU}&=&(\ep (\boldsymbol u), \ep (\boldsymbol w))+\lambda(\div \boldsymbol u, \div \boldsymbol w),\label{eq:norm1}\\
(\boldsymbol v, \boldsymbol z)_{\bV}&=&(R^{-1}_p \boldsymbol v,\boldsymbol z)+\lambda (\div \boldsymbol v, \div \boldsymbol z),\label{eq:norm2}\\
(p, q)_P&=&\lambda^{-1} (p,q).\label{eq:norm3}
\end{eqnarray}
Then the coercivity of $(\ep (\boldsymbol u), \ep (\boldsymbol u)) +\lambda (\div \bm u, \div \bm u)+(R_p^{-1}\bm v, \bm v)$ on the kernel set 
$$
\bm Z=\{(\bm u, \bm v)\in \bm U\times\bm V: (\div \bm u, q)+(\div\bm  v, q)=0, \forall q\in P\}
$$
can be verified by direct computation.
However, in this case the %inf-sup condition from Lemma \ref{Hdiv:inf-sup}
%\begin{equation}
%\inf_{q_h\in P_h}\sup_{v_h\in \bm V_h}\frac{(\div \bm v_h, q_h)}{\|\bm v_h\|_V\|q_h\|_P}\geq \beta_d,
%\end{equation}
%does not {hold for all $\beta_d>0$ independent of $R_p^{-1}$}.
{
$H(\div)$ inf-sup condition
\begin{equation}\label{eq:h_div_inf_sup}
\inf_{q\in P}\sup_{\bm{v}\in \bm{V}}\frac{(\operatorname{div}\bm{v},q)}{\|\bm{v}\|_{\boldsymbol V}\|q\|_P}\geq \beta_{v}>0
\end{equation}
fails if $R_p^{-1} \gg \lambda$, for instance, $\lambda \eqsim 1$ and $R_p^{-1} \gg 1$.
}

Hence, in order to obtain the inf-sup condition
for $(\div \bm u, q)+(\div \bm v, q)$, namely, 
\begin{equation}\label{eq:hdiv_Stokes_inf_sup}
\inf_{q\in P}\sup_{(\bm u, \bm v)\in \bm U\times \bm V}\frac{(\div \bm u, q)+(\div \bm v, q)}{(\|\bm u\|_{\bm U}+\|\bm v\|_{\boldsymbol V})\|q\|_P}\geq \beta_{su}>0,
\end{equation}
the Stokes inf-sup condition {from} Lemma \ref{Stokes:inf-sup} has to be satisfied at the discrete level, as we can see by
choosing $(\bm u, \bm v)=(\bm u, 0)\in \bm U\times \bm V$ in \eqref{eq:hdiv_Stokes_inf_sup}.

{From the above observation, we conclude that} we need to choose a proper space for the approximation of the
displacement field $\bm u$, even if $\lambda$ is small, in order to satisfy the Stokes inf-sup condition stated
in Lemma~\ref{Stokes:inf-sup} {at the} discrete level.
This shows that the discretization of \eqref{eq:8a}--\eqref{eq:8c} using on the spaces $P_1\times RT_0\times P_0$
can not be uniformly stable with respect to all model parameters!
\end{remark}

Therefore, in our paper, we use $H(\div)$ conforming spaces such as $BDM, RT$ and $BDFM$ to replace $P_1$.
Further, {in} this way, we can preserve the divergence condition exactly, which means the fluid mass conservation
is fully preserved. {Details will be presented in the following Section~\ref{sec:uni_stab_disc_model}.}

\section{Uniformly stable discretizations of the model}\label{sec:uni_stab_disc_model}

There are various discretizations that meet the requirements for the proof of the full
parameter-robust stability that will be presented in this section. {They include} conforming
as well as nonconforming methods. In general, whenever $U_h/P_h$ is a Stokes-stable pair
and $V_h/P_h$ satisfies the $H(\div)$ inf-sup condition similar to~\eqref{inf-sup}, the norm
that we have proposed in Section \ref{sec:par_rob_stab_model} allows to prove the full
parameter-robust stability
%using a proof that follows the same ideas as used in
{using similar arguments as in}
the proof of Theorem \ref{Dis:Stability}. To give two popular examples, the triplet $CR/RT_0/P_0$
together with stabilization does result in a parameter-robust stable discretization of the Biot
model if the norms are defined as in Section \ref{sec:par_rob_stab_model}. The same is true
for the conforming discretization based on the spaces $P_2/RT_0/P_0$. However, these finite
element methods can not preserve the fluid mass conservation exactly and globally although
they can achieve parameter-robustness. 

Therefore, in this section, motivated by {the works} \cite{cockburn2007note, hong2016robust, honguniformly},
we propose discretizations of the Biot's model problem~\eqref{eq:8a}--\eqref{eq:8c}. {These discretizations}
preserve the divergence condition (namely equation \eqref{eq:5c}) exactly and globally, which means an exact
conservation of the fluid mass. Furthermore, the discretizations are also locking-free when the Lam\'e parameter
$\lambda$ tends to $\infty$. First we introduce some notations.

\subsection{Preliminaries and notation}
{By $\mathcal{T}_h$ we denote} a shape-regular triangulation of mesh-size $h$ of
the domain $\Omega$ into triangles $\{K\}$. We further denote by
$\mathcal{E}_h^{I}$ the set of all interior edges (or faces) of  $\mathcal{T}_h$ and by
$\mathcal{E}_h^{B}$ the set of all boundary edges (or faces); we set
$\mathcal{E}_h=\mathcal{E}_h^{I}\cup \mathcal{E}_h^{B}$.

For $s\geq 1$, we define
$$
H^s(\mathcal{T}_h)=\{\phi\in L^2(\Omega), \mbox{ such that } \phi|_K\in H^s(K) \mbox{ for all } K\in \mathcal{T}_h \}.
$$

The vector functions are represented column-wise. 

{As we consider discontinuous Galerkin (DG) discretizations}, we define some trace operators next.
Let $e = \partial K_1 \cap \partial K_2$ be the common boundary (interface) of two subdomains $K_1$
and $K_2$ in $\mathcal{T}_h$ , and $\bm n_1$ and $\bm n_2$ be unit normal vectors to $e$ pointing to the
exterior of $K_1$ and $K_2$, respectively. For any edge (or face) $e
\in \mathcal{E}_h^{I}$ and a scalar $q\in H^1(\mathcal{T}_h)$, vector $\bm v \in H^1(\mathcal{T}_h)^d$ and tensor
$\bm \tau \in H^1(\mathcal{T}_h)^{d\times d}$, we define the
averages
\begin{equation*}
\begin{split}
\{\bm v\} &=\frac{1}{2}(\bm v|_{\partial K_1\cap e}\cdot \bm n_1-\bm
v|_{\partial K_2\cap e}\cdot \bm n_2), \quad 
\{\bm \tau\}=\frac{1}{2}(\bm \tau|_{\partial K_1\cap
e} \bm n_1-\bm \tau|_{\partial K_2\cap e} \bm n_2),
\end{split}
\end{equation*}
and jumps
\begin{equation*}
[q]=q|_{\partial K_1\cap e}-q|_{\partial K_2\cap e},\quad
[\bm v]=\bm v|_{\partial K_1\cap e}-\bm v|_{\partial K_2\cap e},\quad
 \Lbrack\bm v\Rbrack=\bm v|_{\partial K_1\cap e}\odot \bm n_1+\bm v|_{\partial K_2\cap e}\odot \bm n_2,
\end{equation*}
where $\bm v \odot \bm n=\frac{1}{2}(\bm{v} \bm{n}^T+\bm n \bm v^T)$ is the 
symmetric part of the tensor product of $\bm v$ and $\bm n$.

When $e \in  \mathcal{E}_h^{B}$ then the above quantities are defined as
\[
\{\bm v\}=\bm v |_{e}\cdot \bm n,\quad
\{\bm \tau\}=\bm \tau|_{e}\bm n, \quad
[q]=q|_{e}, ~~ [\bm v]=\bm v|_{e}, \quad  \Lbrack\bm v\Rbrack=\bm v|_{e}\odot \bm n.
\]
If $\bm n_K$ is the outward unit normal to $\partial K$, it is easy to check that
\begin{equation}\label{eq:68}
\sum_{K\in \mathcal{T}_h}\int_{\partial K}\bm v\cdot \bm n_K q ds=\sum_{e\in \mathcal{E}_h}\int_{e}\{\bm v\} [q] ds ,\quad\mbox{for all}\quad \bm v\in H(\div ;\Omega),\quad\mbox{for all}\quad q\in H^1(\mathcal{T}_h).
\end{equation}
Also, for $\bm \tau \in H^1(\Omega)^{d\times d}$ and for all $\bm v\in H^1(\mathcal{T}_h)^d$, we have
\begin{equation}\label{eq:69}
\sum_{K\in \mathcal{T}_h}\int_{\partial K}(\bm \tau\bm n_K)\cdot  \bm vds=\sum_{e\in \mathcal{E}_h}\int_{e}\{\bm \tau\}\cdot [\bm v] ds.
\end{equation}

The finite element spaces are denoted by
$$
\bm U_h=\{\bm u \in H(\div ;\Omega):\bm u|_K \in \bm U(K),~K \in \mathcal{T}_h;~ \bm u \cdot
\bm n=0~\hbox{on}~\partial \Omega\},
$$
$$
\bm V_h=\{\bm v \in H(\div ;\Omega):\bm v|_K \in \bm V(K),~K \in \mathcal{T}_h;~ \bm v \cdot
\bm n=0~\hbox{on}~\partial \Omega\},
$$
$$
P_h=\{q \in L^2(\Omega):q|_K \in Q(K),~K \in \mathcal{T}_h; ~\int_{\Omega}q dx=0\}.
$$
{The discretizations that we consider here, define the local spaces $\bm U(K)/\bm V(K)/Q(K)$
via the triplets $RT_{l}(K)/RT_{l-1}(K)/P_{l-1}(K)$, or $BDM_l(K)/RT_{l-1}(K)/P_{l-1}(K)$, or
$BDFM_l(K)/RT_{l-1}(K)/P_{l-1}(K)$. Note that for all these choices the important 
condition $\div  \bm U(K)=\div  \bm V(K)=Q(K)$ is satisfied.}

We recall the basic approximation properties of these spaces: for all
$K\in \mathcal{T}_h$ and for all $\bm u \in H^s(K)^d$, there exists $\bm u_I\in
\bm U(K)$ such that
\begin{equation}\label{eq:70}
\|\bm u-\bm u_I\|_{0,K}+h_K|\bm u-\bm u_I|_{1,K}+h_K^2|\bm u-\bm
u_I|_{2,K}
\leq C h_K^s|\bm u|_{s,K}, ~2\le s \leq l+1.
\end{equation}
\subsection{DG discretization}
We note that according to the definition of $\bm U_h$, the normal component of any $\bm u\in \bm U_h$ is continuous
on the internal edges and vanishes on the boundary edges. Therefore, by splitting a vector  $\bm u\in \bm U_h$ into
its  normal and tangential components $\bm u_n$ and $\bm u_t$
\begin{equation}\label{eq:71}
\bm u_n:=(\bm u\cdot \bm n)\bm n,\quad \bm u_t:=\bm u- \bm u_n,
\end{equation}
we have 
\begin{equation}\label{eq:72}
\mbox{for all}\quad e\in \mathcal{E}_h~~\int_e[\bm u_n]\cdot \bm \tau ds=0,\quad\mbox{for all}\quad \bm \tau\in H^1(\mathcal{T}_h)^d, \bm u\in \bm U_h,
\end{equation}
implying that 
\begin{equation}\label{eq:73}
\mbox{for all}\quad e\in \mathcal{E}_h~~\int_e[\bm u]\cdot \bm \tau ds=\int_e[\bm v_t]\cdot \bm \tau ds=0,\quad\mbox{for all}\quad \bm  \tau\in H^1(\mathcal{T}_h)^d,\bm u \in \bm U_h.
\end{equation}
A direct computation 
shows that
\begin{equation}\label{eq:74}
\Lbrack \bm u_t\Rbrack:\Lbrack \bm w_t\Rbrack=\frac12[\bm u_t] \cdot [\bm w_t].
\end{equation}

Therefore, the discretization of the variational problem~\eqref{eq:8a}--\eqref{eq:8c} is given as follows.
Find $(\boldsymbol u_h, \boldsymbol v_h, p_h)\in \boldsymbol U_h\times\boldsymbol V_h\times P_h$, such that for
any $(\boldsymbol w_h, \boldsymbol z_h, q_h)\in \boldsymbol U_h\times\boldsymbol V_h\times P_h$
\begin{subequations}\label{eq:75-77}
\begin{eqnarray}
 a_h(\bm u_h,\bm w_h) +\lambda ( \div \boldsymbol  u_h, \div \boldsymbol  w_h)- ( p_h, \div \boldsymbol  w_h)&=&(\boldsymbol f, \boldsymbol w_h),\label{eq:75}\\
  (R_p^{-1}\boldsymbol  v_h, \boldsymbol  z_h)- (p_h, \div \boldsymbol  z_h)&=&0, \label{eq:76} \\
-(\div \boldsymbol u_h,q_h)  -(\div \boldsymbol v_h,q_h)  -\alpha_p  (p_h,q_h)&=&(g,q_h), \label{eq:77}
\end{eqnarray}
\end{subequations}
where
\begin{eqnarray}
a_h(\bm u,\bm w)&=&\label{78}
\sum _{K \in \mathcal{T}_h} \int_K \varepsilon(\bm{u}) :
\varepsilon(\bm{w}) dx-\sum_{e \in \mathcal{E}_h} \int_e \{\varepsilon(\bm{u})\} \cdot [\bm w_t] ds\\
&&\nonumber-\sum _{e \in \mathcal{E}_h} \int_e \{\varepsilon(\bm{w})\} \cdot [\bm u_t]ds+\sum _{e
\in \mathcal{E}_h} \int_e \eta h_e^{-1}[ \bm u_t] \cdot [\bm w_t] ds,
\end{eqnarray}
and $\eta $ is {a stabilization} parameter which is independent of $h, \lambda,R_p^{-1}, \alpha_p$.

For any $\bm u\in  H^1(\mathcal{T}_h)^d$, we introduce the mesh dependent norms:
\begin{eqnarray*}
\|\bm{u}\|_h^2&=&\sum _{K \in \mathcal{T}_h} 
\|\varepsilon(\bm{u})\|_{0,K}^2+\sum _{e \in \mathcal{E}_h} h_e^{-1}\|[ \bm u_t]\|_{0,e}^2, \\
\|\bm u\|_{1,h}^2&=&\sum _{K \in \mathcal{T}_h} \|\nabla\bm{u}\|_{0,K}^2+\sum _{e \in \mathcal{E}_h} h_e^{-1}\|[ \bm{u}_t]\|_{0,e}^2,\\
\end{eqnarray*}
Next, for $\bm u\in H^2(\mathcal{T}_h)^d$, we define the ``DG''-norm
\begin{equation}\label{DGnorm}
\|\bm u\|^2_{DG}=\sum _{K \in \mathcal{T}_h} \|\nabla\bm{u}\|_{0,K}^2+\sum _{e \in \mathcal{E}_h} h_e^{-1}\|[ \bm{u}_t]\|_{0,e}^2+\sum _{K \in \mathcal{T}_h}h_K^2|\bm{u}|^2_{2,K},
\end{equation}
and, finally, the mesh-dependent norm $\|\bm \cdot \|_{\bm U_h}$ by
\begin{equation}\label{U_hnorm}
\|\bm u\|^2_{\bm U_h}=\|\bm u\|^2_{DG}+\lambda \|\div \bm u\|^2.
\end{equation}
We now summarize several results on well-posedness and approximation
properties of the DG formulation, {see, e.g.~\cite{hong2016robust, honguniformly}.}

\begin{itemize}
\item From the discrete version of Korn's inequality we have that the norms $\|\cdot\|_{DG}$,
$\|\cdot\|_h$, and $\|\cdot\|_{1,h}$ are equivalent on $\bm U_h$, namely,
\begin{equation}\label{Korninequality}
\|\bm{u}\|_{DG}\eqsim  \|\bm{u}\|_h\eqsim\|\bm u\|_{1,h},\quad\mbox{for all}\quad~\bm u \in \bm U_h .
\end{equation}
\item The bilinear form $a_h(\cdot,\cdot)$,
introduced in~\eqref{78} is continuous and we have
\begin{eqnarray}\label{continuity:a_h}
|a_h(\bm u,\bm w)|&\lesssim& \| \bm u  \|_{DG}  \| \bm w  \|_{DG},\quad\mbox{for all}\quad \bm u,~\bm w\in H^2(\mathcal{T}_h)^d.
\end{eqnarray}
\item For our choice of the finite element spaces $\bm{V}_h$ and {$P_h$} we
have the following inf-sup conditions
%$b_h(\cdot,\cdot)$
\begin{equation}\label{inf-sup}
  \inf_{q_h\in P_h}\sup_{\bm{u}_h\in \bm{U}_h}\frac{(\operatorname{div}\bm{u}_h,q_h)}{\|\bm{u}_h\|_{1,h}\|q_h\|}\geq \beta_{sd}>0,\,\,\,\,\,\,\, \inf_{q_h\in P_h}\sup_{\bm{v}_h\in \bm{V}_h}\frac{(\operatorname{div}\bm{v}_h,q_h)}{\|\bm{v}_h\|_{\div}\|q_h\|}\geq \beta_{dd}>0,
\end{equation}
where $\beta_{sd}$ and $\beta_{dd}$ are constant independent of the parameters $\lambda, R_p^{-1}, \alpha_p$ and the mesh size $h$.
\item We also have that $a_h(\cdot,\cdot)$ is coercive, and the proof
  of this fact parallels the proofs of similar results.
\begin{equation}\label{coercivity:a_h}
a_h(\bm{u}_h,\bm{u}_h)\geq \alpha_a \|\bm{u}_h\|^2_h,\quad\mbox{for all}\quad~\bm{u}_h\in\bm{U}_h,
\end{equation}
where $\alpha_a$ is a positive constant independent of parameters $\lambda, R_p^{-1}, \alpha_p$ and the mesh size  $h$.
\end{itemize}

Related to the discrete problem~\eqref{eq:75}-\eqref{eq:77} we introduce the bilinear form 
\begin{equation}\label{eq:79}
\begin{split}
A_h((\boldsymbol u_h,  \boldsymbol v_h, p_h), (\boldsymbol w_h,  \boldsymbol z_h, q_h))=&a_h(\bm u_h,\bm v_h)+\lambda ( \div \boldsymbol  u_h, \div \boldsymbol  w_h)- ( p_h, \div \boldsymbol  w_h)+
  (R_p^{-1}\boldsymbol  v_h, \boldsymbol  z_h)\\
  &- (p_h, \div \boldsymbol  z_h)-(\div \boldsymbol u_h,q_h)  -(\div \boldsymbol v_h,q_h)  -\alpha_p  (p_h,q_h).
\end{split}
\end{equation}
{In view of the definitions of the} norms $\|\cdot\|_{\bm U_h},\|\cdot\|_{\bm V}$ and $\|\cdot\|_P$, the boundedness of the $A_h((\boldsymbol u_h,  \boldsymbol v_h, p_h), (\boldsymbol w_h,  \boldsymbol z_h, q_h))$
is obvious, and we come to our second main result.
\begin{theorem}\label{Dis:Stability}
There exits a constant $\beta_0>0$ independent of  the parameters $\lambda,~R_p^{-1},~\alpha_p$ and
the mesh size~$h$, such that 
\begin{equation}\label{eq:80}
\inf_{(\boldsymbol u_h,  \boldsymbol v_h, p_h)\in \boldsymbol U_h\times\boldsymbol V_h\times P_h} \sup_{(\boldsymbol w_h,\boldsymbol z_h, q_h)\in \boldsymbol U_h\times\boldsymbol V_h\times P_h}
\frac{A_h((\boldsymbol u_h,  \boldsymbol v_h, p_h), (\boldsymbol w_h,  \boldsymbol z_h, q_h))}
{(\|\boldsymbol u_h\|_{\bm U_h}+\|\boldsymbol v_h\|_{\boldsymbol V}+\|p_h\|_{P})(\|\boldsymbol w_h\|_{\bm U_h}+\|\boldsymbol z_h\|_{\boldsymbol V}+\|q_h\|_{P})}\geq \beta_0.
\end{equation}
\end{theorem}
\begin{proof}
Case  \uppercase\expandafter{\romannumeral 1}:
\begin{equation}\label{eq:81}
\rho=\min\{\lambda,R^ {-1}_p\}=\lambda, ~\hbox{hence} ~\lambda\leq R_p^{-1}, ~\gamma^{-1}\leq \rho=\lambda.
\end{equation}
For any $(\boldsymbol u_h,  \boldsymbol v_h, p_h)\in \boldsymbol U_h\times\boldsymbol V_h\times P_h$, by the first inequality in  \eqref{inf-sup}, there exists 
\begin{equation}\label{eq:82}
\boldsymbol u_{h,0}\in \boldsymbol U_h,~\hbox{such that}~\div \boldsymbol u_{h,0}=\frac{1}{\sqrt \lambda} p_h,~ \|\boldsymbol u_{h,0}\|_{1,h}\leq \beta_{sd}^{-1}\frac{1}{\sqrt \lambda} \|p_h\|.
\end{equation}
Choose 
\begin{equation}\label{eq:83}
\boldsymbol w_h=\delta \boldsymbol u_h-\frac{1}{\sqrt \lambda} \boldsymbol u_{h,0}, \boldsymbol z_h=\delta \boldsymbol v_h, q_h= -\delta p_h-\gamma^{-1}\div \boldsymbol v_h,
\end{equation}
where the constant $\delta$ will be determined later.

We verify first the boundedness of $(\boldsymbol w_h,  \boldsymbol z_h, q_h)$ by $(\boldsymbol u_h,  \boldsymbol v_h, p_h)$.

By \eqref{eq:82}, the equivalence between norms $\|\cdot\|_{DG}$ and $\|\cdot\|_{1,h}$, namely \eqref{Korninequality},
and noting that $\gamma^{-1}\leq \rho=\lambda$, we have 
\begin{equation}\label{eq:84}
\begin{split}
\|\frac{1}{\sqrt \lambda} \boldsymbol u_{h,0}\|^2_{\boldsymbol U_h}&=\|\frac{1}{\sqrt \lambda}\boldsymbol u_{h,0}\|^2_{DG}+\lambda (\div \frac{1}{\sqrt \lambda} \boldsymbol u_{h,0}, \div \frac{1}{\sqrt \lambda} \boldsymbol u_{h,0})\\
&\leq C_0\|\frac{1}{\sqrt \lambda} \boldsymbol u_{h,0}\|_{1,h}^2+(\div \boldsymbol u_{h,0}, \div \boldsymbol u_{h,0})~~~ (\hbox{by}~~ \eqref{Korninequality})\\
&\leq \frac{1}{\lambda}C_0 \beta^{-2}_{sd} \frac{1}{\lambda}\|p\|^2+
\frac{1}{\lambda}(p,p)\leq (\frac{1}{\lambda} C_0\beta^{-2}_{sd}+1)\frac{1}{\lambda} \|p_h\|^2\\
&\leq (\frac{1}{\lambda} C_0\beta^{-2}_{sd}+1) \gamma \|p_h\|^2=(\frac{1}{\lambda} C_0\beta^{-2}_{sd}+1) (p_h,p_h)_P.
\end{split}
\end{equation}
Therefore, by taking into account that $1\leq \lambda$ we get for $\boldsymbol w_{h}$ the estimate
\begin{equation}\label{eq:85}
\|\boldsymbol w_h\|_{\boldsymbol U_h}\leq \delta \|\boldsymbol u_h\|_{\boldsymbol U_h}+\sqrt{\lambda ^{-1} C_0\beta^{-2}_{sd}+1} \|p_h\|_P\leq  \delta \|\boldsymbol u_h\|_{\boldsymbol U_h}+ \sqrt{C_0\beta^{-2}_{sd}+1} \|p_h\|_P.
\end{equation}
Obviously we have $\|\boldsymbol z_h\|_{\boldsymbol V}= \delta\|\boldsymbol v_h\|_{\boldsymbol V}$
{and it remains to bound} $q_h$.
From
\begin{equation}\label{eq:86}
(\gamma^{-1} \div \boldsymbol v_h, \gamma^{-1} \div \boldsymbol v_h)_P=\gamma (\gamma^{-1} \div \boldsymbol v_h, \gamma^{-1} \div \boldsymbol v_h)=\gamma^{-1} (\div \boldsymbol v_h, \div \boldsymbol v_h)\leq (\boldsymbol v_h, \boldsymbol  v_h)_{\boldsymbol V}
\end{equation}
it follows that
\begin{equation}\label{eq:87}
\|q_h\|_P\leq \delta \|p_h\|_P+\|\boldsymbol  v_h\|_{\boldsymbol V}.
\end{equation}
Next we {establish} the coercivity of $A_h((\boldsymbol u_h,  \boldsymbol v_h, p_h), (\boldsymbol w_h,  \boldsymbol z_h, q_h))$.

Using the definition of~$(\boldsymbol w_h,  \boldsymbol z_h, q_h)$ and \eqref{eq:82}, we find
\begin{eqnarray*}
&&A_h((\boldsymbol u_h,  \boldsymbol v_h, p_h), (\boldsymbol w_h,  \boldsymbol z_h, q_h))\\
&=&a_h(\bm{u}_h,\bm{w}_h)+\lambda ( \div \boldsymbol  u_h, \div \boldsymbol  w_h)- ( p_h, \div \boldsymbol  w_h)+
  (R_p^{-1}\boldsymbol  v_h, \boldsymbol  z_h)- (p_h, \div \boldsymbol  z_h)\\
&&-(\div \boldsymbol u_h,q_h)  -(\div \boldsymbol v_h,q_h)  -\alpha_p  (p_h,q_h)\\
&=&a_h(\bm{u}_h,\delta\boldsymbol  u_h-\frac{1}{\sqrt \lambda}\boldsymbol u_{h,0}) +\lambda ( \div \boldsymbol  u_h, \delta\div \boldsymbol  u_h-\frac{1}{\sqrt \lambda}\div \boldsymbol u_{h,0})\\
&&-( p_h, \delta\div \boldsymbol  u_h-\frac{1}{\sqrt \lambda}\div \boldsymbol u_{h,0})+(R_p^{-1}\boldsymbol  v_h, \delta\boldsymbol  v_h)- (p_h,  \delta\div \boldsymbol  v_h)\\
&&-(\div \boldsymbol u_h, -\delta p_h-\gamma^{-1}\div \boldsymbol v_h)-(\div \boldsymbol v_h,-\delta p_h-\gamma^{-1}\div \boldsymbol v_h)-\alpha_p  (p_h,-\delta p_h-\gamma^{-1}\div \boldsymbol v_h)\\
&=&\delta a_h(\bm{u}_h,\boldsymbol  u_h)-\frac{1}{\sqrt \lambda}a_h(\bm{u}_h,\boldsymbol u_{h,0}) +\delta \lambda ( \div \boldsymbol  u_h, \div \boldsymbol  u_h)-{\sqrt \lambda}(\div \boldsymbol  u_h,\div \boldsymbol u_{h,0} )\\
&&- \delta(p_h,\div \boldsymbol  u_h)+\frac{1}{\sqrt \lambda}(p_h,\div \boldsymbol u_{h,0})
+\delta(R_p^{-1}\boldsymbol  v_h, \boldsymbol  v_h)- \delta(p_h,  \div \boldsymbol  v_h)+\delta(\div \boldsymbol u_h, p_h)\\
&&+\gamma^{-1}(\div \boldsymbol u_h, \div \boldsymbol v_h)+\delta(\div \boldsymbol v_h, p_h)+\gamma^{-1}(\div \boldsymbol v_h,\div \boldsymbol v_h)+\delta \alpha_p  (p_h, p_h) + \alpha_p(p_h, \gamma^{-1}\div \boldsymbol v_h)\\
&=&\delta a_h(\bm{u}_h,\boldsymbol  u_h)-\frac{1}{\sqrt \lambda}a_h(\bm{u}_h,\boldsymbol u_{h,0})+\delta \lambda ( \div \boldsymbol  u_h, \div \boldsymbol  u_h)-(\div \boldsymbol  u_h, p_h)+\frac{1}{\lambda}(p_h, p_h)~~~ (\hbox{by}~\eqref{eq:82})\\
&&+\delta(R_p^{-1}\boldsymbol  v_h, \boldsymbol  v_h)+\gamma^{-1}(\div \boldsymbol u_h, \div \boldsymbol v_h)+\gamma^{-1}(\div \boldsymbol v_h,\div \boldsymbol v_h)+\delta \alpha_p  (p_h, p_h) + \alpha_p(p_h, \gamma^{-1}\div \boldsymbol v_h)
\end{eqnarray*}
{Next we apply Cauchy's inequality, use the coercivity and the continuity of $a_h(\cdot,\cdot)$, the equivalence of the
norms $\|\cdot\|_{DG}$ and $\|\cdot\|_{1,h}$, and \eqref{eq:82}, to get 
\begin{eqnarray*}
&&A_h((\boldsymbol u_h,  \boldsymbol v_h, p_h), (\boldsymbol w_h,  \boldsymbol z_h, p_h))\\
&=&\delta a_h(\bm{u}_h,\boldsymbol  u_h)-\frac{1}{\sqrt \lambda}a_h(\bm{u}_h,\boldsymbol u_{h,0})
+\delta \lambda ( \div \boldsymbol  u_h, \div \boldsymbol  u_h)-(\div \boldsymbol  u_h, p_h)+\frac{1}{\lambda}(p_h, p_h)\\
&&+\delta(R_p^{-1}\boldsymbol  v_h, \boldsymbol  v_h)+\gamma^{-1}(\div \boldsymbol u_h, \div \boldsymbol v_h)
+\gamma^{-1}(\div \boldsymbol v_h,\div \boldsymbol v_h)+\delta \alpha_p  (p_h, p_h) + \alpha_p(p_h, \gamma^{-1}\div \boldsymbol v_h)\\
&\geq&\delta a_h(\bm{u}_h,\boldsymbol  u_h)-\frac{1}{2}\frac{1}{\sqrt \lambda}\epsilon_1^{-1}a_h(\bm{u}_h,\boldsymbol  u_h)
-\frac{1}{2}\frac{1}{\sqrt \lambda}\epsilon_1a_h(\bm{u}_{h,0},\boldsymbol  u_{h,0}) +\delta \lambda ( \div \boldsymbol  u_h, \div \boldsymbol  u_h)\\
&&-\frac{1}{2}\epsilon_2^{-1}\lambda(\div \boldsymbol  u_h, \div \boldsymbol  u_h)-\frac{1}{2}\epsilon_2\frac{1}{\lambda}(p_h, p_h)
+\frac{1}{\lambda}(p_h, p_h)+\delta(R_p^{-1}\boldsymbol  v_h, \boldsymbol  v_h)\\
&&-\frac{1}{2}\epsilon_3^{-1}\gamma^{-1}(\div \boldsymbol u_h, \div \boldsymbol u_h)
-\frac{1}{2}\epsilon_3\gamma^{-1}(\div \boldsymbol v_h, \div \boldsymbol v_h)+\gamma^{-1}(\div \boldsymbol v_h,\div \boldsymbol v_h)  \\
&&+\delta \alpha_p  (p_h, p_h) -\frac{1}{2}\epsilon_4\gamma^{-1}(\div \boldsymbol v_h, \div \boldsymbol v_h)
-\frac{1}{2}\epsilon_4^{-1}\alpha_p^2\gamma^{-1}(p_h, p_h)\\
&\geq&(\delta-\frac{1}{2}\frac{1}{\sqrt \lambda}\epsilon_1^{-1})\alpha_a\|\bm{u}_h\|_{DG}^2
+(\delta -\frac{1}{2}\epsilon_2^{-1})\lambda ( \div \boldsymbol  u_h, \div \boldsymbol  u_h)
-\frac{1}{2}\epsilon_3^{-1}\gamma^{-1}(\div \boldsymbol u_h, \div \boldsymbol u_h)\\
&&+\delta(R_p^{-1}\boldsymbol  v_h, \boldsymbol  v_h)
+(1-\frac{1}{2}\epsilon_3-\frac{1}{2}\epsilon_4)\gamma^{-1}(\div \boldsymbol v_h,\div \boldsymbol v_h)
\qquad \qquad \qquad \qquad \qquad \,
(\hbox{by}~\eqref{coercivity:a_h}~~\hbox{and}~~\eqref{Korninequality})\\
&&+(1-\frac{1}{2}\frac{1}{\sqrt \lambda}\epsilon_1C^2_1C_0^2\beta_{sd}^{-2}
-\frac{1}{2}\epsilon_2)\frac{1}{\lambda}(p_h, p_h)+(\delta -\frac{1}{2}\epsilon_4^{-1}\alpha_p\gamma^{-1})\alpha_p  (p_h, p_h).
\qquad
(\hbox{by}~\eqref{eq:82},\eqref{continuity:a_h}~~\hbox{and}~~\eqref{Korninequality})
\end{eqnarray*}
}
{
Now letting $\epsilon_1=\frac{1}{2}C^{-2}_1C_0^{-2}\beta_{sd}^2, \epsilon_2=\epsilon_3=\epsilon_4=\frac{1}{2}$, and noting that $\rho^{-1}\leq \gamma, \lambda\geq \gamma^{-1}>0$, we obtain
}
{
\begin{eqnarray*}
A((\boldsymbol u_h,  \boldsymbol v_h, p_h), (\boldsymbol w_h,  \boldsymbol z_h, p_h))
&\geq&(\delta-\frac{1}{\sqrt \lambda} C^2_1C^2_0\beta_{sd}^{-2})\alpha_a\|\bm{u}_h\|_{DG}^2+(\delta -2)\lambda ( \div \boldsymbol  u_h, \div \boldsymbol  u_h)+\delta(R_p^{-1}\boldsymbol  v_h, \boldsymbol  v_h)\\
&&+\frac{1}{2}\gamma^{-1}(\div \boldsymbol v_h,\div \boldsymbol v_h)+\frac{1}{2}\frac{1}{\lambda}(p_h, p_h)+(\delta -\alpha_p\gamma^{-1})\alpha_p  (p_h, p_h)
\end{eqnarray*}
Next, setting $\delta:=\max\{C^2_1C^2_0\beta_{sd}^{-2}+\frac{1}{2}\alpha^{-1}_a, 2+\frac{1}{2}\}$ and
noting that $\alpha_p\leq \gamma, \lambda \geq 1$,} we derive the coervicity estimate
\begin{eqnarray*}
A((\boldsymbol u_h,  \boldsymbol v_h, p_h), (\boldsymbol w_h,  \boldsymbol z_h, p_h))
&\geq&\frac{1}{2}\|\bm{u}_h\|_{DG}^2 +\frac{1}{2}\lambda ( \div \boldsymbol  u_h, \div \boldsymbol  u_h)+\frac{1}{2}(R_p^{-1}\boldsymbol  v_h, \boldsymbol  v_h)\\
&&+\frac{1}{2}\gamma^{-1}(\div \boldsymbol v_h,\div \boldsymbol v_h)+\frac{1}{2}\frac{1}{\lambda}(p_h, p_h)+\frac{1}{2}\alpha_p  (p_h, p_h)\\
&\geq&\frac{1}{2}\big(\|\boldsymbol u_h\|^2_{\boldsymbol U_h}+\|\boldsymbol  v_h\|^2_{\boldsymbol V}+\|p_h\|^2_P\big).
\end{eqnarray*}
Case \uppercase\expandafter{\romannumeral 2}:
\begin{equation}\label{eq:89}
\rho=\min\{\lambda,R^ {-1}_p\}=R^ {-1}_p, ~\hbox{hence} ~\lambda\geq R_p^{-1}, ~\gamma^{-1}\leq \rho=R^ {-1}_p.
\end{equation}
For any $(\boldsymbol u_h,  \boldsymbol v_h, p_h)\in \boldsymbol U_h\times\boldsymbol V_h\times P_h$, by the second inequality in \eqref{inf-sup}, there exists 
\begin{equation}\label{eq:90}
\boldsymbol v_{h,0}\in \boldsymbol V_h,~\hbox{such that}~\div \boldsymbol v_{h,0}={\sqrt {R_p}} p_h,~ \|\boldsymbol v_{h,0}\|_{\div}\leq \beta_{dd}^{-1}{\sqrt {R_p}} \|p_h\|.
\end{equation}
Choose 
\begin{equation}\label{eq:91}
\boldsymbol w_h=\delta \boldsymbol u_h, \boldsymbol z_h=\delta \boldsymbol v_h-\sqrt {R_p}\boldsymbol v_{h,0}, q_h= -\delta p_h-\gamma^{-1}\div \boldsymbol v_h,
\end{equation}
where $\delta$ is a constant which will be specified later.

Again we first verify the boundedness of
$(\boldsymbol w_h,  \boldsymbol z_h, q_h)$ by $(\boldsymbol u_h,  \boldsymbol v_h, p_h)$.
We note that $\|\boldsymbol w_h\|_{\boldsymbol U_h}=\delta \|\boldsymbol u_h\|_{\boldsymbol U_h}$.

From \eqref{eq:90} and noting that $\gamma^{-1}\leq \rho=R_p^{-1}$ we have 
\begin{equation}\label{eq:92}
\begin{split}
({\sqrt R_p} \boldsymbol v_{h,0}, {\sqrt R_p} \boldsymbol v_{h,0})_{\boldsymbol V}
&=(R_p^{-1} {\sqrt R_p}\boldsymbol v_{h,0}), {\sqrt R_p}\boldsymbol v_{h,0})
+\gamma^{-1}(\div\sqrt {R_p} \boldsymbol v_{h,0}, \div\sqrt {R_p} \boldsymbol v_{h,0})\\
&\leq(\boldsymbol v_{h,0},  \boldsymbol v_{h,0})+(\div \boldsymbol v_{h,0}, \div \boldsymbol v_{h,0}) \\
&\leq \beta^{-2}_{dd} R_p\|p_h\|^2\leq \beta^{-2}_{dd} \gamma\|p_h\|^2\leq\beta^{-2}_{dd}(p_h,p_h)_P.
\end{split}
\end{equation}
Hence, we get the boundedness of $\boldsymbol z_h$, that is
\begin{equation}\label{eq:93}
\|\boldsymbol z_h\|_{\boldsymbol V}\leq \delta \|\boldsymbol v_h\|_{\boldsymbol V}+\beta^{-1}_{dd} \|p_h\|_P.
\end{equation}
Again we have the boundedness for $q_h$ {according to} \eqref{eq:87}.

In what follows we show the coercivity of
$A_h((\boldsymbol u_h,  \boldsymbol v_h, p_h), (\boldsymbol w_h, \boldsymbol z_h, q_h))$
in Case \uppercase\expandafter{\romannumeral 2}.

\noindent
Using the definition of $(\boldsymbol w_h,  \boldsymbol z_h, q_h)$ and \eqref{eq:90}, we find
\begin{eqnarray*}
&&A_h((\boldsymbol u_h,  \boldsymbol v_h, p_h), (\boldsymbol w_h, \boldsymbol z_h, q_h))\\
&=&a_h(\bm u_h, \bm w_h) +\lambda ( \div \boldsymbol  u_h, \div \boldsymbol  w_h)- ( p_h, \div \boldsymbol  w_h)+
  (R_p^{-1}\boldsymbol  v_h, \boldsymbol  z_h)- (p_h, \div \boldsymbol  z_h)\\
&&-(\div \boldsymbol u_h,q_h)  -(\div \boldsymbol v_h,q_h)  -\alpha_p  (p_h,q_h)\\
&=&a_h(\bm u_h, \delta\bm u_h) +\lambda ( \div \boldsymbol  u_h, \delta\div \boldsymbol  u_h )
-( p_h, \delta\div \boldsymbol  u_h)+(R_p^{-1}\boldsymbol  v_h, \delta\boldsymbol  v_h-\sqrt{R_p} \boldsymbol v_{h,0})\\
&&-(p_h,  \delta\div \boldsymbol  v_h-\sqrt{R_p} \div\boldsymbol v_{h,0})-(\div \boldsymbol u_h, -\delta p_h-\gamma^{-1}\div \boldsymbol v_h)\\
&&-(\div \boldsymbol v_h,-\delta p_h-\gamma^{-1}\div \boldsymbol v_h)  -\alpha_p  (p_h,-\delta p_h-\gamma^{-1}\div \boldsymbol v_h)\\
&=&\delta a_h(\bm u_h, \bm u_h) +\delta \lambda ( \div \boldsymbol  u_h, \div \boldsymbol  u_h)-\delta(p_h,\div \boldsymbol  u_h)
+\delta(R_p^{-1}\boldsymbol  v_h, \boldsymbol  v_h)-({R^{-1/2}_p}\boldsymbol  v_h, \boldsymbol v_{h,0})\\
&&-\delta(p_h,  \div \boldsymbol  v_h)+(p_h, \sqrt{R_p} \div\boldsymbol v_{h,0})+\delta(\div \boldsymbol u_h, p_h)
+\gamma^{-1}(\div \boldsymbol u_h, \div \boldsymbol v_h)\\
&&+\delta(\div \boldsymbol v_h, p_h)+\gamma^{-1}(\div \boldsymbol v_h,\div \boldsymbol v_h)+\delta \alpha_p  (p_h, p_h)
+ \alpha_p(p_h, \gamma^{-1}\div \boldsymbol v_h)\\
&=&\delta a_h(\bm u_h, \bm u_h) +\delta \lambda ( \div \boldsymbol  u_h, \div \boldsymbol  u_h)
+\delta(R_p^{-1}\boldsymbol  v_h, \boldsymbol  v_h)-({R^{-1/2}_p}\boldsymbol  v_h, \boldsymbol v_{h,0})+(p_h, R_pp_h)
\qquad \qquad \quad \ \ \, \mbox{by } \eqref{eq:90}\\
&&+\gamma^{-1}(\div \boldsymbol u_h, \div \boldsymbol v_h)+\gamma^{-1}(\div \boldsymbol v_h,\div \boldsymbol v_h)  +\delta \alpha_p  (p_h, p_h) + \alpha_p(p_h, \gamma^{-1}\div \boldsymbol v_h).
\end{eqnarray*}
{Next, we apply Cauchy's inequality, use~\eqref{eq:90}, and the coercivity of $a_h(\cdot,\cdot)$, to get
\begin{eqnarray*}
&&A_h((\boldsymbol u_h,  \boldsymbol v_h, p_h), (\boldsymbol w_h, \boldsymbol z_h, q_h))\\
&=&\delta a_h(\bm u_h, \bm u_h) +\delta \lambda ( \div \boldsymbol  u_h, \div \boldsymbol  u_h)+\delta(R_p^{-1}\boldsymbol  v_h, \boldsymbol  v_h)
-({R^{-1/2}_p}\boldsymbol  v_h, \boldsymbol v_{h,0})+(p_h, R_pp_h)\\
&&+\gamma^{-1}(\div \boldsymbol u_h, \div \boldsymbol v_h)+\gamma^{-1}(\div \boldsymbol v_h,\div \boldsymbol v_h)
+\delta \alpha_p  (p_h, p_h) + \alpha_p(p_h, \gamma^{-1}\div \boldsymbol v_h)\\
&\geq&\delta a_h(\bm u_h, \bm u_h)+\delta \lambda ( \div \boldsymbol  u_h, \div \boldsymbol  u_h)+\delta(R_p^{-1}\boldsymbol  v_h, \boldsymbol  v_h)
-\frac{1}{2}\epsilon_1^{-1}({R^{-1}_p}\boldsymbol  v_h, \boldsymbol v_h)-\frac{1}{2}\epsilon_1(\boldsymbol  v_{h,0}, \boldsymbol v_{h,0})\\
&&+({R_p} p_h, p_h)-\frac{1}{2}\epsilon_2^{-1}\gamma^{-1}(\div \boldsymbol u_h, \div \boldsymbol u_h)
-\frac{1}{2}\epsilon_2\gamma^{-1}(\div \boldsymbol v_h, \div \boldsymbol v_h)\\
&&+\gamma^{-1}(\div \boldsymbol v_h,\div \boldsymbol v_h)  +\delta \alpha_p  (p_h, p_h)
-\frac{1}{2}\epsilon_3\gamma^{-1}(\div \boldsymbol v_h, \div \boldsymbol v_h) -\frac{1}{2}\epsilon_
3^{-1}\alpha_p^2\gamma^{-1}(p_h, p_h)\\
&\geq&\delta \alpha_a\|\bm u_h\|_{DG}^2 +\delta \lambda ( \div \boldsymbol  u_h, \div \boldsymbol  u_h)
-\frac{1}{2}\epsilon_2^{-1}\gamma^{-1}(\div \boldsymbol u_h, \div \boldsymbol u_h)
+(\delta-\frac{1}{2}\epsilon_1^{-1})({R^{-1}_p}\boldsymbol  v_h, \boldsymbol v_h)
\quad \mbox{by } \eqref{coercivity:a_h} \mbox{, } \eqref{eq:90} \\
&&+(1-\frac{1}{2}\epsilon_2-\frac{1}{2}\epsilon_3)\gamma^{-1}(\div \boldsymbol v_h, \div \boldsymbol v_h)
+(1-\frac{1}{2}\epsilon_1\beta^{-2}_{dd})(R_p p_h, p_h)+(\delta -\frac{1}{2}\epsilon_3^{-1}\alpha_p\gamma^{-1})\alpha_p  (p_h, p_h).
%~~\,\,\,\,\hbox{by}\,\, \eqref{eq:90}\\
\end{eqnarray*}
}
Now, letting $\epsilon_1=\beta_{dd}^2, \epsilon_2=\epsilon_3=\frac{1}{2}$, and noting that $\rho^{-1}\leq \gamma, \lambda\geq \rho\geq \gamma^{-1}>0$, we obtain
\begin{eqnarray*}
&&A_h((\boldsymbol u_h,  \boldsymbol v_h, p_h), (\boldsymbol w_h, \boldsymbol z_h, q_h))\\
&\geq&\delta\alpha_a\|\bm u_h\|_{DG}^2 +(\delta-1)\lambda (\div \boldsymbol  u_h, \div \boldsymbol  u_h)+(\delta-\frac{1}{2}\beta_d^{-2})(R_p^{-1}\boldsymbol  v_h, \boldsymbol  v_h)\\
&&+\frac{1}{2}\gamma^{-1}(\div \boldsymbol v_h,\div \boldsymbol v_h)+\frac{1}{2}(R_pp_h, p_h)+(\delta -\alpha_p\gamma^{-1})\alpha_p  (p_h, p_h)
\end{eqnarray*}
Finally, we choose $\delta:=\max\{\frac{1}{2}\alpha_a^{-1}, \frac{1}{2}\beta_{dd}^{-2}+\frac{1}{2}, 1+\frac{1}{2}\}$,
note that $\alpha_p\leq \gamma$, and conclude the coervicity of the bilinear form, i.e.,
\begin{eqnarray*}%\label{eq:91}
A_h((\boldsymbol u_h,  \boldsymbol v_h, p_h), (\boldsymbol w_h, \boldsymbol z_h, q_h))
&\geq&\frac{1}{2}\|\bm u_h\|_{DG}^2 +\frac{1}{2}\lambda (\div \boldsymbol  u_h, \div \boldsymbol  u_h)+\frac{1}{2}(R_p^{-1}\boldsymbol  v_h, \boldsymbol  v_h)\\
&&+\frac{1}{2}\gamma^{-1}(\div \boldsymbol v_h,\div \boldsymbol v_h)+\frac{1}{2}(R_pp_h, p_h)+\frac{1}{2}\alpha_p  (p_h, p_h)\\
&\geq&\frac{1}{2}\big(\|\boldsymbol u_h\|^2_{\boldsymbol U_h}+\|\boldsymbol  v_h\|^2_{\boldsymbol V}+\|p_h\|^2_P\big).
\end{eqnarray*}
This completes the proof.
\end{proof}
From the above theorem, we {get the following stability estimate.}
\begin{corollary}\label{eq:92}
Let $(\boldsymbol u_h,  \boldsymbol v_h, p_h)\in \boldsymbol U_h\times \boldsymbol V_h\times P_h$ be the solution of \eqref{eq:75}-\eqref{eq:77},
then we have the estimate
\begin{equation}
\|\boldsymbol u_h\|_{\boldsymbol U_h}+\|\boldsymbol  v_h\|_{\boldsymbol V}+\|p_h\|_P\leq C_2 (\|\boldsymbol f\|_{\boldsymbol U_h^*}+\|g\|_{P^*}),
\end{equation} 
where $\|\boldsymbol f\|_{\boldsymbol U_h^*}=
\sup\limits_{\boldsymbol w_h\in \boldsymbol U_h}\frac{(\boldsymbol f, \boldsymbol w_h)}{\|\boldsymbol w_h\|_{\boldsymbol U_h}},\|g\|_{P^*}=\sup\limits_{q_h\in P_h}\frac{(g,q_h)}{\|q_h\|_P}$
and $C_2$ is a constant independent of $\lambda, R_p^{-1}, \alpha_p$ and mesh size $h$.
\end{corollary}
%{From Theorem \ref{Dis:Stability}, a parameter-robust block preconditioner can be designed correspondingly. That will be our ongoing work on the way.}
{\begin{remark}
Denote by $A_h$ the operator induced by the bilinear form \eqref{eq:79}, namely 
\begin{equation}\label{operator:A_h}
A_h:=\left[\begin{array}{ccc}
-\div_h \ep_h-\lambda \nabla_h \div_h & 0 & \nabla_h \\
0 & R_p^{-1}I_h & \nabla_h \\
-\div_h & - \div_h & -\alpha_{p} I_h  \end{array}\right],
\end{equation} and
define 
\begin{equation}\label{Preconditioner:Bh}
B_h:=\left[\begin{array}{ccc}
(-\div_h \ep_h-\lambda \nabla_h \div_h)^{-1} & 0 & 0 \\
0 & (R_p^{-1}I_h+ \gamma^{-1}\nabla_h \div_h)^{-1}& 0 \\
0& 0& (\gamma I_h )^{-1}  \end{array}\right].
\end{equation}
Then due to the theory presented in \cite{Mardal2011preconditioning}, Theorem~\ref{Dis:Stability} implies
that the norm-equivalent (canonical) block-diagonal preconditioner $B_h$ for $A_h$ is parameter-robust,
which means {that} the condition number $\kappa(B_hA_h)$ is uniformly bounded with respect to the
parameters $\lambda, R_p^{-1}, \alpha_p$ in the ranges~\eqref{parameter:range} and with respect to the
mesh size $h$.
\end{remark}
}

%\begin{remark}
%In practice, relevant parameters in the soft tissue of the central
%nervous system are Young’s modulus of  $1 − 60$ kPa, Poisson ratio from $0.3$ to
%almost $0.5$, and the permeability is  $10^{-14}-10^{-16} m^2$. In
%geophysics, Young’s modulus is typically in the order of GPa, Poisson ratio $0.1-0.3$,
%while the permeability may vary from approximately $10^{-9} $ to $10^{-21}$. Relations of Young’s modulus $E$, Poisson ratio $\nu$ and the two elastic moduli $\mu$, are $\mu=E/2(1+\nu)$ and $\lambda=E\nu/(1+\nu)(1-2\nu)$. Consequently, $\mu$ and $\lambda$
%in the ranges of $300-500$ MPa and $100-500$ MPa, respectively, in geoscience
%applications, whereas corresponding numbers are in the ranges $300-2000
%$ Pa and $500-10^6$ Pa in neurological applications. In discrete case, always the time step size $\tau$ is very small. And hence, in our notation, $\lambda$ can be comparable to $R_p^{-1}$, or can be much smaller than $R_p^{-1}$. Further, $\alpha_p$ can be very small or comparable to 
%$\frac{1}{\lambda}$.
%\end{remark}

\section{Error estimates}\label{sec:error_estimates}
In this section, we derive the error estimates {that follow from the results presented} in Section~\ref{sec:uni_stab_disc_model}.
{Let    
$\Pi_B^{\div}: H^1(\Omega)^d\mapsto \bm U_h$ be the canonical interpolation operator. We also denote the
$L^2$-projection on $P_h$ by~$Q_h$. The following Lemma, see \cite{hong2016robust}, summarizes some of the properties of
$\Pi_B^{ \div}$ and $Q_h$ needed for our proof.
\begin{lemma}\label{stable:Interpolation}
For all $\bm w \in H^1(K)^d$ we have
\begin{eqnarray*}
 \div \Pi_B^{ \div}=Q_h  \div\;;\quad
|\Pi_B^{ \div} \bm w|_{1,K} \lesssim |\bm w|_{1,K};~  \|\bm w-\Pi_B^{ \div} \bm w\|^2_{0,\partial K}\lesssim h_K
  |\bm w|^2_{1,K}.
  %\quad 
  %\| \div(\bm w-\Pi_h^{ \div} \bm w)\|_{-1}\lesssim h_K
  %\| \div\bm w\|,
\end{eqnarray*}
%where $\|r\|_{-1} = \sup_{\chi\in H^1}\frac{(\chi,r)}{\|\chi\|_1}$. 
\end{lemma}
}
\begin{theorem}\label{error0}
Let $(\bm u,\bm v, p)$ be the solution of~\eqref{eq:8a}--\eqref{eq:8c} and $(\bm u_h,\bm v_h, p_h)$
be the solution of \eqref{eq:75}--\eqref{eq:77}. Then the error estimates
\begin{equation} \label{eq:erroruv}
\|\bm u-\bm u_h\|_{\bm U_h}+\|\bm v-\bm v_h\|_{\bm V}\leq C_{e,u}  \inf\limits_{\boldsymbol w_h\in \boldsymbol U_h, \bm z_h\in \bm V_h}\Big(\|\bm u-\bm w_h\|_{\bm U_h}+\|\bm v-\bm z_h\|_{\bm V}\Big),
\end{equation}
and 
\begin{equation} \label{eq:errorp}
\|p-p_h\|_P\leq C_{e,p}  \inf\limits_{\boldsymbol w_h\in \boldsymbol U_h, \bm z_h\in \bm V_h, q_h\in P_h}\Big(\|\bm u-\bm w_h\|_{\bm U_h}+\|\bm v-\bm z_h\|_{\bm V}+\|p-q_h\|_P\Big),
\end{equation}
hold, where $C_{e,u}, C_{e,p}$ are constants independent of $\lambda, R_p^{-1}, \alpha_p$ and the mesh size $h$.
\end{theorem}
\begin{proof}
Subtracting \eqref{eq:75}--\eqref{eq:77} from \eqref{eq:8a}--\eqref{eq:8c} and noting the consistency of $ a_h(\cdot,\cdot)$, we have that for any $(\boldsymbol w_h, \boldsymbol z_h, q_h)\in \boldsymbol U_h\times\boldsymbol V_h\times P_h$
\begin{eqnarray}
 a_h(\bm u-\bm u_h,\bm w_h) +\lambda ( \div (\boldsymbol u- \boldsymbol u_h), \div \boldsymbol  w_h)- ( (p-p_h), \div \boldsymbol  w_h)&=&0,\label{eq:error1}\\
  (R_p^{-1}(\boldsymbol v- \boldsymbol v_h), \boldsymbol  z_h)- (p-p_h, \div \boldsymbol  z_h)&=& 0, \label{eq:error2} \\
-( \div (\boldsymbol u- \boldsymbol u_h),q_h)  -( \div (\boldsymbol v- \boldsymbol v_h),q_h)  -\alpha_p  (p-p_h,q_h)&=&0. \label{eq:error3}
\end{eqnarray}
%We only consider the space choice of $BDM_l/RT_{l-1}/P_{l-1}$, the other cases are similar. 

{Let $\boldsymbol u_I=\Pi_{B}^{div}\boldsymbol u\in \boldsymbol U_h, p_I=Q_h p\in P_h$.} 
%where $\Pi_{B}^{div}$
%is the canonical interpolation operator from $\boldsymbol U$ to $\boldsymbol U_h$,  and $Q_h$ is the $L^2$ orthogonal
%projection from $L^2(\Omega)$ to $P_h$. 
Now for arbitrary $\boldsymbol v_I\in \bm V_h$,
%$\boldsymbol v_I=\pi_{R}^{div} \boldsymbol v \in \boldsymbol V_h$. $\pi_{R}^{div}$ is the projection operator introduced by Winther from $\boldsymbol V$ to $\boldsymbol V_h$
from~\eqref{eq:error1}--\eqref{eq:error3}, noting that $\div\Pi_{B}^{div}=Q_h\div$
and $\div \boldsymbol U_h=\div \boldsymbol V_h=P_h$, we conclude
\begin{eqnarray}
 a_h(\bm u_I-\bm u_h,\bm w_h) +\lambda ( \div (\boldsymbol u_I- \boldsymbol u_h), \div \boldsymbol  w_h)- ( (p_I-p_h), \div \boldsymbol  w_h)&=&a_h(\bm u_I-\bm u,\bm w_h),\label{eq:error4}\\
  (R_p^{-1}(\boldsymbol v_I- \boldsymbol v_h), \boldsymbol  z_h)- (p_I-p_h, \div \boldsymbol  z_h)&=& (R_p^{-1}(\boldsymbol v_I- \boldsymbol v), \boldsymbol  z_h), \label{eq:error5} \\
-( \div (\boldsymbol u_I- \boldsymbol u_h),q_h)  -( \div (\boldsymbol v_I- \boldsymbol v_h),q_h)  -\alpha_p  (p_I-p_h,q_h)&=&-(\div (\boldsymbol v_I- \boldsymbol v),q_h). \label{eq:error6}
\end{eqnarray}
{Next, since} $(\bm u_I-\bm u_h)\in \bm U_h, (\bm v_I-\bm v_h)\in V_h, (p_I-p_h)\in P_h$,
by the stability result~\eqref{eq:80} for the discrete problem~\eqref{eq:75}--\eqref{eq:77}, we obtain
\begin{equation*}%\label{eq:error7}
\|\bm u_I-\bm u_h\|_{\bm U_h}+\|\bm v_I-\bm v_h\|_{\bm V}\leq C_e \Big( \sup\limits_{\boldsymbol w_h\in \boldsymbol U_h}\frac{a_h(\bm u_I-\bm u,\bm w_h)}{\|\bm w_h\|_{\bm U_h}}+\sup\limits_{\boldsymbol z_h\in \boldsymbol V_h}\frac{(R_p^{-1}(\boldsymbol v_I- \boldsymbol v), \boldsymbol  z_h)}{\|\bm z_h\|_V}+\sup\limits_{q_h\in P_h}\frac{(\div (\boldsymbol v- \boldsymbol v_I),q_h)}{\|q_h\|_P}\Big),
\end{equation*}
and 
\begin{equation*}%\label{eq:error8}
\|p_I-p_h\|_P\leq C_e  \Big(\sup\limits_{\boldsymbol w_h\in \boldsymbol U_h}\frac{a_h(\bm u_I-\bm u,\bm w_h)}{\|\bm w_h\|_{\bm U_h}}+\sup\limits_{\boldsymbol z_h\in \boldsymbol V_h}\frac{(R_p^{-1}(\boldsymbol v_I- \boldsymbol v), \boldsymbol  z_h)}{\|\bm z_h\|_V}+\sup\limits_{q_h\in P_h}\frac{(\div (\boldsymbol v- \boldsymbol v_I),q_h)}{\|q_h\|_P}\Big).
\end{equation*}
Hence, using the boundedness of $a_h(\cdot, \cdot)$, {the second inequality in Lemma \ref{stable:Interpolation}},
and triangle inequality, {we arrive at} 
\begin{equation} \label{eq:error9}
\|\bm u-\bm u_h\|_{\bm U_h}+\|\bm v-\bm v_h\|_{\bm V}\leq C_{e,u}  \inf\limits_{\boldsymbol w_h\in \boldsymbol U_h, \bm z_h\in \bm V_h}\Big(\|\bm u-\bm w_h\|_{\bm U_h}+\|\bm v-\bm z_h\|_{\bm V}\Big),
\end{equation}
and 
\begin{equation} \label{eq:error10}
\|p-p_h\|_P\leq C_{e,p}  \inf\limits_{\boldsymbol w_h\in \boldsymbol U_h, \bm z_h\in \bm V_h, q_h\in P_h}\Big(\|\bm u-\bm w_h\|_{\bm U_h}+\|\bm v-\bm z_h\|_{\bm V}+\|p-q_h\|_P\Big).
\end{equation}
\end{proof}
\begin{remark}
From the above theorem, we can see that the discretizations are locking-free.
%when the material goes to nearly impressible.
\end{remark}

\section{Conclusions}\label{conclusion}
{This paper presents the stability analysis of a classical three-field formulation of Biot's consolidation
model where the unknown variables are the displacements, fluid flux (Darcy velocity), and pore pressure. 
Specific parameter-dependent norms provide the key to establish the parameter-robust stability of the continuous
problem.
{This allows for the construction of a parameter-robust block diagonal preconditioner in the framework of operator preconditioning.}
Discretizations {that fully preserve} the fluid mass conservation are designed. Further, both discrete parameter-robust
stability and locking-free error estimates are proved. 
%That will be our ongoing work on the way.
}

%\section{$P_2\times RT_0\times P_0$ is uniform stable for $\tau$}
%$\lambda =1,c_{up}=1, 0\le c_pp=c_{up}/{\lambda},K=1,\tau $ goes to zero.
%The norm has to be changed and a counterexample can be given to show that the $H(\div)$ inf-sup condition does not work any more. But it is still stable.
%\section{$P_2\times RT_0\times P_0$ is uniform for $\tau$ and $\lambda$}
%$\lambda$ goes to infinity, $c_{up}=1, 0\le c_{pp}=c_{up}/{\lambda},K=1,\tau $ goes to zero.
%The norm has to be changed such that the Stokes inf-sup condition to be satisfied.
%\section{$P_2\times RT_0\times P_0$ is uniform for the scaled norms}
%$1\le\lambda$ goes to infinity, $0<c_{up}<\infty, 0\le c_{pp}=c_{up}/{\lambda},K=1,\tau $ goes to zero.

\bibliographystyle{unsrt}
\bibliography{reference_pro}

\end{document}